\documentclass{amsart}

\usepackage[latin1]{inputenc}
\usepackage{amsfonts}
\usepackage{amsmath}
\usepackage{amsthm}
\usepackage{amssymb}
\usepackage{latexsym}
\usepackage{enumerate}
\usepackage{tikz}

\newtheorem{theorem}{Theorem}[section]
\newtheorem{lemma}[theorem]{Lemma}
\newtheorem{proposition}[theorem]{Proposition}

\newtheorem{question}[theorem]{Question}

\newtheorem{definition}[theorem]{Definition}

\numberwithin{equation}{section}

\newtheoremstyle{TheoremNum}
  {\topsep}{\topsep}              %%% space between body and thm
  {\itshape}                      %%% Thm body font
  {}                              %%% Indent amount (empty = no indent)
  {\bfseries}                     %%% Thm head font
  {.}                             %%% Punctuation after thm head
  { }                             %%% Space after thm head
  {\thmname{#1}\thmnote{ \bfseries #3}}%%% Thm head spec
\theoremstyle{TheoremNum}
\newtheorem{reptheorem}{Theorem}
\begin{document}
\newcommand{\pp}{\mathfrak{p}}
\newcommand{\cc}{\mathfrak{c}}
\newcommand{\N}{\mathbb{N}}
\newcommand{\Q}{\mathbb{Q}}
\newcommand{\C}{\mathbb{C}}
\newcommand{\Z}{\mathbb{Z}}
\newcommand{\R}{\mathbb{R}}
\newcommand{\T}{\mathbb{T}}
\newcommand{\I}{\mathcal{I}}
\newcommand{\J}{\mathcal{J}}
\newcommand{\A}{\mathcal{A}}
\newcommand{\HH}{\mathcal{H}}
\newcommand{\K}{\mathcal{K}}
\newcommand{\B}{\mathcal{B}}
\newcommand{\st}{*}
\newcommand{\PP}{\mathbb{P}}
\newcommand{\SSS}{\mathbb{S}}
\newcommand{\forces}{\Vdash}
\newcommand{\dom}{\text{dom}}
\newcommand{\osc}{\text{osc}}
\newcommand\encircle[1]{%
  \tikz[baseline=(X.base)] 
    \node (X) [draw, shape=circle, inner sep=0] {\strut #1};}

\title[Noncommutative Cantor-Bendixson derivatives]{Noncommutative Cantor-Bendixson derivatives
and scattered $C^*$-algebras}

\author{Saeed Ghasemi}
\address{Institute of Mathematics, Polish Academy of Sciences,
ul. \'Sniadeckich 8,  00-656 Warszawa, Poland}
\email{\texttt{sghasemi@impan.pl}}

\author{Piotr Koszmider}
\address{Institute of Mathematics, Polish Academy of Sciences,
ul. \'Sniadeckich 8,  00-656 Warszawa, Poland}
\email{\texttt{piotr.koszmider@impan.pl}}
\thanks{This research  of the second named author was partially supported by   grant
PVE Ci\^encia sem Fronteiras - CNPq (406239/2013-4).}

%\subjclass[2010]{54D30,  03E35, 46B25}
\begin{abstract} We   analyze the sequence obtained by consecutive applications of 
the Cantor-Bendixson derivative for a noncommutative scattered $C^*$-algebra
$\mathcal A$, using the ideal $\I^{At}(\mathcal A)$ 
generated  by the minimal projections of $\mathcal A$. 
With its help, we present some  fundamental results concerning
scattered $C^*$-algebras, in a manner parallel to the commutative case
of scattered compact or locally compact Hausdorff spaces and
superatomic Boolean algebras. It also allows us to formulate
problems which have motivated  the ``cardinal sequences''  programme 
in the classical topology,  in 
the noncommutative context. This leads to some new constructions
of noncommutative scattered $C^*$-algebras and new open problems. 
In particular, we construct a type $I$ $C^*$-algebra which is the inductive limit
 of  stable ideals $\A_\alpha$, along
an  uncountable limit ordinal $\lambda$, such that $\A_{\alpha+1}/\A_\alpha$
is $*$-isomorphic to the algebra of all compact operators on 
a separable Hilbert space and $\A_{\alpha+1}$ is 
$\sigma$-unital and stable for each $\alpha<\lambda$, but $\A$ is not stable and where all ideals
of $\A$ are of the form $\A_\alpha$. In particular, $\A$ is a nonseparable $C^*$-algebra 
with no  ideal which is maximal among the stable ideals. 
 This answers a question of M. R{\o}rdam in the nonseparable case.
All the above $C^*$-algebras $A_\alpha$s and $A$ satisfy the following
version of the definition of an AF algebra: any finite subset can be approximated
from a finite-dimensional subalgebra.
Two more complex constructions based on the language developed in this paper
are presented in separate papers \cite{psi-space, thin-tall}.
\end{abstract}

\maketitle

\section{Introduction}

\subsection{Background and goals}
The first modern paper on superatomic Boolean algebras (the algebras where every subalgebra has an atom)
was written by Mostowski and Tarski in 1939 (\cite{mostowski-tarski}) while a serious
research on dispersed (or scattered) compact Hausdorff spaces (the spaces where every subspace has a relative isolated point)
can be traced back to Cantor with a substantial abstract result already in 1920 in the paper of
Mazurkiewicz and Sierpi\'nski (\cite{mazurkiewicz-sierpinski}). Since the papers of Day (\cite{day}),
Rudin (\cite{rudin}), Pe\l czy\'nski and Semadeni (\cite{pelczynski-semadeni}), it was realized that
the topics of superatomic Boolean algebras and of scattered compact spaces are the same
topics in two different languages (formally, by the Stone duality,
a compact $K$ is scattered if and only if the Boolean algebra $Clop(K)$, of all clopen subsets of $K$ 
forms a basis for $K$ and  is  superatomic) and that they produce
an interesting class of Banach spaces of the form $C(K)$ with many peculiar features\footnote{Most Banach space theory 
refers to Banach spaces over the reals and in this context $C(K)$ stands for the Banach space of all real valued functions
on a compact Hausdorff space $K$,  most of the results can be transfered verbatim to the case of
complex Banach spaces, for example of continuous complex valued functions. However when we equip $C(K)$
with the multiplication and the involution and talk
about a it as a $C^*$-algebra, by  $C(K)$ we always mean the $C^*$-algebra of  complex valued continuous functions.}.
In fact, already in 1930 J. Schreier used the compact scattered space $K=[0,\omega^\omega]$
to provide an ingenious negative answer to a problem of Banach, if nonisomorphic Banach spaces may have
isomorphic dual spaces (\cite{schreier}).  Later two important classes of Banach spaces 
emerged by generalizing these $C(K)$s, namely, Asplund spaces (those Banach spaces where
separable subspaces have separable duals) and Asplund generated spaces (those which
admit a dense range linear bounded operator from an Asplund space). These classes assumed
fundamental roles in the theory of vector valued measures and renorming theory of Banach spaces
(\cite{vector-measures, dgz}), due to the Radon-Nikodym property in the dual and
the Frechet differentiability properties.

At the same time the research related to the dual pair of superatomic Boolean algebras and scattered
compact spaces underwent  dramatic set-theoretic  developments, especially related to the use
of additional set-theoretic assumptions.  For example, Ostaszewski constructed his
hereditary separable non hereditarily Lindel\"of space assuming Jensen's $\diamondsuit$
principle (\cite{ostaszewski}), Tall obtaind a separable nonmetrizable normal Moore 
space assuming $\pp>\omega_1$ (\cite{tall}).
The interaction with Banach space theory very much included these developments.
For example, under the continuum hypothesis Kunen's scattered compact space $K$,
in the form of $C(K)$, provided an example of a nonseparable Banach space with no
uncountable biorthogonal system (\cite{negrepontis}). Pol used $C(K)$ for $K$ equal to the ladder system space to
answer a question of Corson, whether all weakly Lindel\"of Banach spaces
are weakly compactly generated (\cite{pol}). Many other equally fundamental constructions
followed (\cite{simon}, \cite{biorthogonal-hajek}, \cite{christina}).

It was Tomiyama in 1963 who first addressed the question which now would be
expressed as ``what are the
$C^*$-algebras which are Asplund Banach spaces?" (\cite{tomiyama}). 
This was done just in the separable case, and later developed in the general
case by Jensen (\cite{jensen1, jensen2, jensen3}) and Huraya (\cite{huruya}). Further understanding
was obtained by Wojtaszczyk in \cite{wojtaszczyk}, Chu in \cite{chu-rn}, where the Radon-Nikodym
property of the duals of scattered $C^*$-algebras was proved. In \cite{chu-banach-saks},
where the weak Banach-Saks property was investigated in the context of scattered $C^*$-algebras
and in \cite{chu-crossed} where crossed products of scattered $C^*$-algebras were investigated.
Morover Lin proved that scattered $C^*$-algebras are approximately finite in \cite{lin}.
Recent papers \cite{kusuda-af}, \cite{kusuda-crossed} of Kusuda look at scattered
$C^*$-algebras in the context of AF-algebras and crossed products.

As often happens,   the noncommutative theory was developed based on the 
noncommutative analogues of
the phenomena taking place  in the dual or the bidual rather than the algebra itself. This is
because of the lack of a well-behaved version of the $K$ for $C(K)$, in the general case. After all, $K$
naturally can be considered as the subspace of the the dual Banach space of $C(K)$, the space 
of Radon measures on $K$.
The purpose of this paper is
to present an alternative approach which focuses on  the phenomena taking place in the 
noncommutative analogues of
 Boolean algebra of clopen sets of $K$, in
other words the projections of $C(K)$. Through the Stone duality, they 
are quite topological, referring to the level of $K$. 
We aim at  three goals: a) to present the usefulness
of  a natural notion of Cantor-Bendixson derivative
of a $C^*$-algebra by obtaining elementary and  parallel (to
the commutative case) new proofs of many results concerning scattered
$C^*$-algebras;  
b) to motivate a collection of problems and constructions of
$C^*$-algebras parallel to the commutative well established topic of
cardinal sequences of superatomic Boolean algebras or scattered (locally) compact Hausdorff spaces
(more constructions can be found in our two
subsequent papers \cite{psi-space} and \cite{thin-tall});
c) to illustrate the accessibility of the noncommutative realm to 
 the reasonings of set-theoretic topological nature.

\subsection{Cantor-Bendixson derivatives}
In Section 2 we review fundamental facts and some proofs
concerning the commutative case, i.e., the superatomic Boolean algebras and
scattered compact spaces, to which we will be appealing
for intuitions in the following sections. 

Section 3 is devoted to developing elementary apparatus allowing us
to work with scattered $C^*$-algebras in a manner parallel to the
commutative case of superatomic
Boolean algebras. The first surprising difference between the literatures concerning the commutative and the noncommutative
case is the way one decomposes the commutative and noncommutative scattered objects. 
In the Boolean algebras and compact spaces 
the main operations leading to the decomposition is the Cantor-Bendixson derivative. In general, these  kind of decompositions are usually called the ``composition series" for a $C^*$-algebra (see IV.1.1.10 of \cite{Blackadar}).
On one hand, since scattered $C^*$-algebras are type $I$ (or GCR) (see 
Proposition \ref{real-rank-zero} (3) or \cite{jensen1}),  they are subject to a
unique composition series  of their ideals such that each consecutive quotient is the largest type $I_0$ (or CCR) 
subalgebra (see e.g, IV.1.1.12 of \cite{Blackadar} and Theorem 1.5.5 of \cite{invitation}).
 On the other hand, a generally ``finer'' and non-unique composition series for   scattered $C^*$-algebras
 is presented in Theorem 2 of \cite{jensen2}, where each consecutive quotient is elementary
  (see also Theorem \ref{theorem1} (\ref{ec-sequence-old})). 
The Cantor-Bendixson decomposition sequence (Theorem \ref{theorem1} (\ref{ec-sequence-new}))
is a third way. It is unique but may have terms which do not appear in the GCR composition series
as well as there could be terms of the GCR composition series which do
not appear at all in the Cantor-Bendixson sequence (Proposition \ref{not-gcr}). Recall that 
in the case of the Banach space theory manifestation of scattered objects,  
a version of the Cantor-Bendixson derivative (in the form of 
 the Szlenk index) turned out
to be a fundamental combinatorial tool as well (see e.g. \cite{lancien}). So, we believe that
one should investigate  the $C^*$-algebra version of the Cantor-Bendixson decomposition.
In our approach the role of Boolean atoms 
(corresponding to isolated points in topology) in $C^*$-algebras is played by 
the following well-known notion of a minimal
projection:

\begin{definition}\label{definition-minimal}
A projection $p$ in a $C^*$-algebra $\mathcal A$  is called minimal if $p\mathcal A p=\C p$.
The set of minimal projections of $\A$ will be denoted by $At(\A)$.
The $*$-subalgebra of $\mathcal A$ generated by the minimal projections of $\mathcal A$
will be denoted $\I^{At}(\mathcal A)$.
\end{definition}

The origin of the notation $\I^{At}(\A)$ is the Boolean algebra notation for the notion 
of an atom (Definition \ref{atoms}).
For example, we have $\chi_{\{x\}}C(K)\chi_{\{x\}}=\C\chi_{\{x\}}$,
for any  isolated point of a compact space $K$, and in fact these are the only minimal
projections in $C(K)$. The right candidate for the
Cantor-Bendixson derivative turns out to be  the mapping from a $C^*$-algebra $\mathcal A$
to its quotient $\mathcal A/\I^{At}(\mathcal A)$ by the ideal $\I^{At}(\mathcal A)$. 
The main observations from Section 3 concerning $\I^{At}(\mathcal A)$ can be summarized as follows:
\vfill
\break
\begin{theorem} Suppose that $\mathcal A$ is a $C^*$-algebra.
\begin{enumerate}
\item $\I^{At}(\mathcal A)$ is an ideal of $\mathcal A$,
\item $\I^{At}(\mathcal A)$ is isomorphic to a subalgebra of the algebra $\mathcal K(\mathcal H)$ of all 
 compact operators on a Hilbert space $\mathcal H$,
\item $\I^{At}(\mathcal A)$ contains all ideals of $\mathcal A$ which
are  isomorphic to a subalgebra of
$\mathcal K(\mathcal H)$ for some Hilbert space $\mathcal H$, 
\item if an ideal $\mathcal I\subseteq \A$ is essential and
isomorphic to a subalgebra of $\mathcal K(\mathcal H)$ for some Hilbert space $\mathcal H$,
then $\mathcal I=\I^{At}(\mathcal A)$.
\end{enumerate}
\end{theorem}
\begin{proof} See Propositions \ref{atoms-are-compact}, \ref{atoms-are-max-ideal}, \ref{essential-compact}.
\end{proof}
This corresponds to the commutative case, when $\A$ is a Boolean
ring, $\mathcal K(\mathcal H)$ for $\mathcal H=\ell_2(\kappa)$
for some cardinal $\kappa$, is replaced by the Boolean ring $Fin(\kappa)$ of finite  subsets
of $\kappa$ and essential ideals are replaced by dense Boolean ideals:
the Boolean ring generated by atoms is an ideal, it is isomorphic to the ring
$Fin(\kappa)$ for some $\kappa$, it contains all ideals of $\A$ isomorphic to $Fin(\lambda)$ for any $\lambda$,
and if an ideal  is dense and isomorphic to $Fin(\kappa)$, it is the ideal generated by all atoms.

It should be noted that one can consider the usual Cantor-Bendixson derivatives (i.e., removing
all isolated points) in the spectrum space $\hat \A$ of a $C^*$-algebra $\A$. 
Moreover as proved by Jensen \cite{jensen2} and exploited by Lin in \cite{lin}, this derivative
makes sense for scattered $C^*$-algebras and, for
example, the Cantor-Bendixson height of the spectrum is the same
as the Cantor-Bendixson height of the algebra in our sense. However, the spectrum of $\mathcal A$
in nontrivial noncommutative cases is not a Hausdorff space, e.g., often (quite
often in the scattered case) given
any two points one is in the closure of the other (cf. Proposition \ref{chain}). 
Moreover the spectrum does not determine the algebra, unlike  in the commutative case
via the Stone duality, for example $\K(\ell_2)\oplus \K(\ell_2)$ and $C(\{1,2\})$ have
the same spectra - two point Hausdorff space.
So in a sense, our approach 
is to transport these derivatives to the algebra level, where it corresponds to
intuitive Boolean notions and more readily becomes a tool for investigating
the $C^*$-algebra. The techniques of section 3 are standard and similar to
the development of the basic theory of CCR and GCR (or type $I$)  algebras (see e.g. \cite{invitation}).

\subsection{Scattered $C^*$-algebras}
We may define scattered $C^*$-algebras in a manner
parallel to the superatomic Boolean algebras:

\begin{definition}A $C^*$-algebra $\A$ is called scattered if 
for every nonzero $C^*$-subalgebra $\B\subseteq \A$, the
ideal $\I^{At}(\B)$ is nonzero.
\end{definition}

It should be added that already in the paper \cite{tomiyama} of Tomiyama the role
of minimal projections in scattered $C^*$-algebras is exploited.
In section 4 we prove
the following six equivalent conditions including  (\ref{ec-sequence-old}) and
 (\ref{ec-spectrum}) known in the $C^*$-algebras literature
to be equivalent to the
traditional definition.

\begin{theorem}[\cite{jensen1, jensen2, wojtaszczyk}]\label{theorem1}
 Suppose that $\mathcal A$ is a $C^*$-algebra.
The following conditions are equivalent:
\begin{enumerate}

\item\label{ec-homomorphic} Every non-zero $*$-homomorphic   image of $\mathcal A$ has a minimal projection.
\item\label{ec-sequence-new} There is an ordinal $ht(\mathcal A)$
and a continuous  increasing sequence of closed ideals $(\I_\alpha)_{\alpha\leq ht(\mathcal A)}$
such that  $\I_0=\{0\}$, $\I_{ht(\mathcal A)}= \mathcal A$ and
$$\I^{At}(\mathcal A /\I_\alpha)=\{[a]_{\I_\alpha}: a\in \I_{\alpha+1}\},$$
for every $\alpha < ht(\mathcal A)$.
\item\label{ec-sequence-old} There is an ordinal $m(\A)$
and a continuous  increasing sequence of ideals $(\J_\alpha)_{\alpha \leq m(A)}$
such that $\J_0=\{0\}$, $\J_{m(\A)} = \mathcal A$ and  $\J_{\alpha+1}/\J_\alpha$ is 
an elementary $C^*$-algebra (i.e., $^*$-isomorphic to the algebra
of all compact operators  on a Hilbert space)
for every $\alpha < m(\A)$.
\item\label{ec-subalgebra} Every non-zero subalgebra of $\mathcal A$ has a minimal projection.
\item\label{ec-no-interval} $\A$ does not contain a copy of the $C^*$-algebra $C_0((0,1])=\{f\in C((0,1]):
\lim_{x\rightarrow 0}f(x)=0\}$.
\item\label{ec-spectrum} The spectrum of every self-adjoint element is countable.
\end{enumerate}
\end{theorem}

The sequence from (\ref{ec-sequence-new}) will be called the Cantor-Bendixson sequence.
Note that unlike the sequence from (\ref{ec-sequence-old}), it is unique.
The ordinal $ht(A)$ is called the Cantor Bendixson height of $\A$.

These conditions have very precise analogs in the commutative  setting  
(Theorem \ref{eq-compact-scattered}), which
become additionally combinatorial due to the fact that all scattered compact spaces are
totally disconnected, i.e., are precisely represented by the Boolean algebras
of their clopen subsets (cf. Theorem  \ref{eq-superatomic}). One can also note that
scattered $C^*$-algebras  possess the noncommutative version
of the zero-dimensionality, i.e., they are of real rank zero (Proposition \ref{real-rank-zero}).

In section 4, after the proof of Theorem
\ref{theorem1} we present several proofs of known properties of scattered $C^*$-algebras
using our equivalences, for example, Kusuda's result  that being scattered is determined by
separable commutative subalgebras or that a $C^*$-algebra is scattered if and
only if all of its subalgebras are AF (\cite{kusuda-af}) in the sense that
every finite set of elements can be approximated from a finite-dimensional subalgebra.
It should be noted that equivalent definitions of separable approximately finite $C^*$-algebras
are no longer equivalent in the nonseparable context: Farah and 
Katsura proved that in general the above version of being AF does not
imply that  a $C^*$-algebra is the direct limit of a directed system of finite dimensional
subalgebras (\cite{farah-katsura}). We adopt the terminology of \cite{farah-katsura}
that is, the above notion is called {\sl locally finite dimensional} (LF). Morover it was
recently proved that the properties LF and AF are not equivalent even for 
 for scattered $C^*$-algebras (\cite{aiau}).

In section 5, we calculate
the Cantor-Bendixson derivatives of the tensor product $\A\otimes\K(\ell_2)$ which is
used extensively in the last section.

Having developed a language analogous to the Cantor-Bendixson derivatives for
locally compact Hausdorff spaces the  next natural  step is to ask if analogous
phenomena happen in the noncommutative context.  The phenomena we focus on
are related to the interaction between the height and the width of locally compact spaces:

\begin{definition}\label{width-height} Suppose that $\A$ is a scattered $C^*$-algebra
with the Cantor-Bendixson sequence $(\mathcal I_\alpha)_{\alpha\leq ht(\A)}$.
The width of $\A$ is the supremum of $\kappa$, where $\mathcal I_{\alpha+1}/\mathcal I_\alpha $
 is isomorphic to  a nondegenerate subalgebra of
$\K(\ell_2(\kappa))$, and it is denoted by $wd(\A)$. 

A scattered $C^*$-algebra $\A$ is called 
$\kappa$-thin-tall ($\kappa$-thin-very tall) if $ht(\A)=\kappa^+$ ($ht(\A)=\kappa^{++}$)
and $wd(\A)=\kappa$. An $\omega$-thin-tall  ($\omega$-thin-very tall) algebra is called thin-tall
 (thin-very tall). 

A scattered $C^*$-algebra $\A$ is called $\kappa$-short-wide if and only if $\I^{At}(\A)\cong\K(\ell_2(\kappa))$,
$\A/\I^{At}(\A)\cong \K(\ell_2(\kappa^+))$ and 
$ht(\A)=2$. An $\omega$-short-wide algebra is called a $\Psi$-algebra.

\end{definition}

The investigations of thin-tall, thin-very tall commutative algebras as well as $\Psi$-algebras
led to many fundamental discoveries in topology and Banach space theory mentioned in the first
part of the introduction (e.g., Ostaszewski's and Kunen's spaces are thin-tall,
Tall's or Simon's are $\Psi$-spaces and an example of Lindenstrauss and Johnson
from \cite{lj} is of the form $C(K)$ for  a $\Psi$-space $K$).  We focus on the case of $\kappa=\omega$ because it is the most
interesting and also because we would like to avoid discussing here constructions that require
additional set-theoretic assumptions. For example, the only known constructions
of commutative $\kappa$-thin-tall algebras for $\kappa>\omega$ use such assumptions (\cite{koepke})
and for $\kappa$-short-wide algebras, it is known that such assumptions
are necessary (Theorem 3.4 of \cite{baumgartner}).  Also it is known that
already the existence of a thin-very tall commutative $C^*$-algebra is independent
of the usual axioms (\cite{just, bs}). 

As we investigate the noncommutative constructions, we need to
impose strong noncommutativity conditions. We consider two such  conditions,
the stability of a $C^*$-algebra $\A$, i.e., the condition that $\A\cong \A\otimes \K(\ell_2)$
(see e.g. \cite{rordam-stable}) and the condition that all the quotients $\I_{\alpha+1}/\I_{\alpha}$
for $\alpha<ht(\A)$
are isomorphic to $\K(\ell_2(\kappa_\alpha))$ for some  cardinal $\kappa_\alpha$,
where $(\I_{\alpha})_{\alpha\leq ht(\A)}$ is the Cantor-Bendixson sequence of $\A$.
The latter condition we call the ``full noncommutativity" (Definition \ref{fully}). Section 6
is devoted to proving some simple observations about this notion, which is
equivalent to the fact that all ideals of the algebra are among the ideals
$\I_{\alpha}$ for $\alpha<ht(\A)$ and that the centers of 
the multiplier algebras of all quotients of $\A$
 (in particular of $\A$ itself) are trivial (\ref{chain}). Also
the Cantor-Bendixson sequence coincides with the GCR-composition
series in fully noncommutative scattered algebras (Proposition \ref{gcr-fully}). Stability does
not imply being fully noncommutative (take e.g. $c_0\otimes \K(\ell_2)$)
but for separable   $C^*$-algebras with no unital quotients the converse implication holds (\ref{fnoncom-stable}).

In the last section 7 we construct  fully noncommutative and stable 
examples of thin-tall and $\Psi$-algebras as well as algebras
of width $\kappa$ and height $\theta$ for any ordinal $\theta<\kappa^+$
and any regular cardinal $\kappa$ (\ref{kappa-going-up}, \ref{exists-thin-tall}, \ref{exists-psi}),
showing that one can obtain 
in the noncommutative context all the constructions of scattered $C^*$-algebras
analogous to the commutative examples that do not require additional set-theoretic assumptions.

The most interesting construction is of the thin-tall scattered $C^*$-algebra, which
as in the commutative case (cf. \cite{juhasz-weiss}, \cite{roitmanhandbook}) requires
some kind of decomposition which is nontrivial at countable limit stages of the
construction.
As an interesting side-product of this technique we obtain a nonstable type $I$ algebra which is an
 uncountable inductive limit
of stable algebras (Theorem \ref{long-thin-tall}). Moreover all the involved
$C^*$-algebras  satisfy the following
version of the definition of an AF algebra called after \cite{farah-katsura}
locally finite dimensional : any finite subset can be approximated
from a finite-dimensional subalgebra.

However, a result of  Hjelmborg and R{\o}rdam (Corollary 4.1. \cite{rordam-h}),
shows that countable inductive limits of stable separable 
(or $\sigma$-unital, more generally) $C^*$-algebras, is again stable. 
This also follows from the model theory fact that being stable
is $\forall\exists$-axiomatizable for separable C*-algebras (Proposition 2.7.7 of \cite{Farah-model-theory}), and 
 that $\forall\exists$-axiomatizability is preserved under taking (countable) inductive limits (see Proposition 2.4.4 of     
\cite{Farah-model-theory}). Therefore, our example shows that these
 results can not be strengthened to the  inductive limits along uncountable ordinals.
Moreover this example shows that there may not exist an 
 ideal which is maximal among all   stable ideals of a $C^*$-algebra. This answers the following
question of R{\o}rdam negatively for (only)  nonseparable $C^*$-algebras.
\begin{question}
{(Question 6.5. \cite{rordam-stable})}  Does every (separable) C*-algebra A have a greatest
stable ideal (i.e., a stable ideal that contains all other stable ideals)?
\end{question}
The argument here is based on the fact that we have
the complete list of ideals in a  fully noncommutative scattered $C^*$-algebra 
 (Lemma \ref{all-ideals}). This construction has been improved in \cite{thin-tall}.

Two further natural questions are whether the consistency results concerning
superatomic Boolean algebras or locally compact scattered spaces 
can be obtained in the fully noncommutative and stable forms and weather
the wider context of noncommutative $C^*$-algebras provides new possibilities.
Perhaps the
most interesting one is if it is consistent that there is a scattered $C^*$-algebra
of countable width and height $\omega_3$. Whether this is possible 
in the commutative case, is a very well known and old open problem.
This leads to a more general question, whether ``behind" any scattered $C^*$-algebra
there is a commutative scattered $C^*$-algebra which ``carries" similar
combinatorics. Here one should mention that
there are nonseparable scattered $C^*$-algebras with no
nonseparable commutative subalgebras (see \cite{tristan}, \cite{akemann-doner}, however  there is
a commutative algebra ``behind" this example) and a very interesting result of Kusuda in \cite{kusuda-crossed},
based on Dauns-Hofmann theorem, shows that a $C^*$-algebra is scattered if and
only if it is GCR and its center is scattered (in most of our examples the center is null). 

One  can also wonder how interesting  such constructions are
from the point of view of $C^*$-algebras. We found two more examples, one of
a thin-tall algebra and the other of a $\Psi$-algebra, which exhibit 
extraordinary behaviors, however, because of their complexity they are
presented in separate papers \cite{psi-space, thin-tall}.

As the paper is intended for readers of diverse backgrounds including set theorists,
classical topologists, Banach spaces theorists and $C^*$-algebraists, many arguments are explained in
details not seen in the papers addressed to a monothematic group.

\subsection{Notation and terminology}
 The notation and the terminology of the paper should be mainly standard.
Initial parts of
two introductory books \cite{invitation} and \cite{murphy} are completely
sufficient as the $C^*$-algebras background. For topology
we use the terminology of \cite{engelking} and for Boolean algebras the ones
of \cite{koppelberg-handbook}, \cite{halmos}. In
particular, we assume familiarity with the Stone
duality. Usually locally compact spaces considered in 
this paper are Hausdorff. The only exception is
the spectrum of a $C^*$-algebra. 
$C_0(X)$  denotes the $C^*$-algebra of complex-valued continuous functions on a locally
compact Hausdorff $X$ vanishing at infinity. If $X$ is compact, we write $C(X)$ for $C_0(X)$.
For a Hilbert space  $\HH$,
 $\B(\HH)$ and $\K(\HH)$ denote the $C^*$-algebras of all bounded operators
on it and all compact operators, respectively. All Hilbert spaces are over the field of complex 
numbers, in particular  the spaces
 $\ell_2(X)$, where $X$ is considered just as a set of indices regardless
of the structure that can be carried by $X$. All ideals in $C^*$-algebras are meant to be two-sided
and closed.
When we talk about a sequence of objects indexed by ordinals, we say
that it is  continuous, if  the terms of the sequence at limit ordinals
are the norm closures of the union of the previous terms, in the case of $C^*$-algebras;
if they are the unions of the previous terms, in the case of Boolean algebras;
and if  they are interesections of the previous terms, in the case of compact spaces.
For a $C^*$-subalgebra $\mathcal A$ of $\mathcal B(\mathcal \HH)$
and a subset $\mathcal D$ of $\mathcal H$, we use the notation
$[\mathcal A \mathcal D]$ to denote the closed subspace of $\mathcal H$
 spanned by the vectors $T\xi$ for $T\in\mathcal A$ and $\xi\in \mathcal D$.
$\mathcal A|\mathcal H$ denotes the set of restrictions of all
elements of $\A$ to an invariant subspace $\HH$. 
The (minimal) unitization of a $C^*$-algebra $\A$ will be denoted by $\widetilde{\A}$.
Masa means a maximal self-adjoint abelian subalgebra and all the ideals considered in this paper are 
closed two-sided ideals. Also $\mathfrak c$ denotes the cardinal continuum.
%CH denotes the  continuum hypothesis.

\section{Review of the commutative case - scattered compact spaces}

We start reviewing the commutative situation by recalling basic facts concerning
Boolean algebras:

\begin{definition}\label{atoms} An atom of a Boolean algebra $A$ is its nonzero element $a$
such that $0\leq b\leq a$ implies $b=0$ or $b=a$ for $b\in A$. 
$I^{At}(A)$ denotes the Boolean ideal generated by the atoms of $A$.
\end{definition}

\begin{theorem}[\cite{mostowski-tarski}, \cite{day}]\label{eq-superatomic} Suppose that $A$ is a Boolean algebra. Then
the following conditions are equivalent:
\begin{enumerate}
\item Every subalgebra of $A$ has an atom.
\item $A$ does not contain any free infinite subalgebra.
\item Every homomorphic image of $A$ is atomic, that is has an atom below 
every nonzero element.
\item Every homomorphic image of $A$ has an atom. 
\item There is an ordinal $ht(A)$
and a continuous  increasing sequence of ideals $(I_\alpha)_{\alpha\leq ht(A)}$
whose union is $A$
such that
$$I^{At}(A/I_\alpha)=\{[a]_{I_\alpha}: a\in I_{\alpha+1}\}$$
holds for each $\alpha<ht(A)$.
\item There is an ordinal $k(A)$,
and a continuous  increasing sequence of ideals $(J_\alpha)_{\alpha<k(A)}$
whose union is $A$
such that for every $\alpha<k(A)$ whenever
$J\subseteq A$ is an ideal of $A$ such that $J_\alpha\subseteq J\subseteq J_{\alpha+1}$,
then $J=J_\alpha$ or $J=J_{\alpha+1}$.
\end{enumerate}
\end{theorem}
We do not include the proof of the theorem. However, in  a sense it is included below, by the Stone duality,
in the equivalences expressed in the topological language in Theorem \ref{eq-compact-scattered}.
One can also find most of it in \cite{roitmanhandbook}. Item (6) is a version of
Proposition 3.5 from \cite{koppelberg-minimal}.

\begin{definition}A Boolean algebra is called superatomic if and only if
each of its homomorphic images is atomic, i.e., it has an atom below
any nonzero element.
\end{definition}

\begin{definition}A compact Hausdorff space is called scattered if and only
if each of its nonempty subsets has a relative isolated point.
\end{definition}

\begin{lemma}\label{0-dim} Suppose that $K$ is a scattered compact Hausdorff space.
Then $K$ is totally disconnected and so $0$-dimensional. Hence it
is the Stone space of a Boolean algebra. 
\end{lemma}
\begin{proof}As connected subspaces have no isolated points, the total disconnectedness
is clear. The fact that for compact Hausdorff spaces the total disconnectedness implies
$0$-dimensionality follows from 6.1.23 of \cite{engelking}.
\end{proof}

\begin{proposition} Suppose that $K$ is a compact Hausdorff and totally disconnected space
and $A$ is a Boolean algebra.
The Boolean algebra of clopen subsets of $K$ is superatomic if and only if
$K$ is scattered. $A$ is superatomic if and only if its Stone space is scattered.
\end{proposition}
\begin{proof} By taking closures, it is clear that every subset of $K$ has an isolated point
if and only if every closed subset of $K$ has an isolated point. The rest follows from
the Stone duality which gives the correspondence between
the closed subsets of a compact space and homomorphic images of the dual Boolean algebra.
\end{proof}

If $X$ is a topological space by  $X'$ we
 will denote the closed subspace of $X$ consisting of  all nonisolated points of $X$. $X'$ 
is called the Cantor-Bendixson derivative of $X$.
Considering a compact Hausdorff scattered $K$, allows us to obtain more equivalent
conditions for a compact space to be scattered, compared to the  Boolean algebraic conditions from 
Theorem \ref{eq-superatomic}:

\begin{theorem}\label{eq-compact-scattered} Suppose that $K$ is a compact Hausdorff space. Then
the following conditions are equivalent:
\begin{enumerate}
\item Every nonempty (closed) subspace of $K$ has an isolated point.

\item There is an ordinal $ht(A)$
and a continuous  decreasing sequence of
closed subspaces  $(K^{(\alpha)})_{\alpha\leq ht(K)}$ of $K$ such that $K^{(0)}=K$, 
$K^{(ht(K))}=\emptyset$ and 
$$K^{(\alpha+1)}=(K^{(\alpha)})'$$
holds for each $\alpha<ht(K)$.
\item There is an ordinal $m(A)$
and a continuous  decreasing sequence of
closed subspaces  $(L^{(\alpha)})_{\alpha\leq m(K)}$ of $K$
such that $L^{(0)}=K$,
$L^{(m(K)))}=\emptyset$ and 
$L^{(\alpha)}\setminus  L^{(\alpha+1)}$ is one point
 for each $\alpha<m(K)$.
\item Every continuous image of $K$ has an isolated point.
\item The range (spectrum) of every $f\in C(K)$ is countable.
\item $K$ does not map onto $[0,1]$.

\end{enumerate}
\end{theorem}
\begin{proof}
(1) $\Rightarrow$ (2) By (1) the procedure described in (2) results in a strictly smaller subspace,
hence at some point we end up with the empty subspace.

(2) $\Rightarrow$ (3) The sets $K^{(\alpha)}\setminus K^{(\alpha+1)}$
are made of isolated points of $K^{(\alpha)}$, so any subspace of $K^{(\alpha)}\setminus K^{(\alpha+1)}$
is open in $K^{(\alpha)}$. In particular we can obtain $ K^{(\alpha+1)}$ from
$K^{(\alpha)}$ by producing a decreasing continuous sequence of closed sets each
obtained from the previous by removing one point. This results in the decomposition 
from (3).

(3) $\Rightarrow$ (4) Let $\phi: K\rightarrow M$ be a surjection,  suppose that $K$ satisfies (3) and
 that $M$ has no isolated points. As in compact spaces infinite sets have (necessarily
nonisolated) accumulation point, the preimage $\phi^{-1}(\{x\})$ of any point
$x\in M$ must contain some nonisolated point of $K$, because otherwise it would be
a finite set of isolated points, which would imply that $x$ is isolated by the fact that $\phi$ is
a closed mapping. This means that $\phi[L^{(\alpha)}]=M$ for every $\alpha\leq m(K)$
which contradicts the emptyness of $L^{(m(L))}$.

(4) $\Rightarrow$ (5)  If the range of some $f\in C(K)$ was uncountable,
then it would contain a closed set with no isolated point by Cantor-Bendixson theorem and so
 it would contain a copy of a Cantor set since subsets of $\C$ are metrizable.
This Cantor set can be mapped onto $[0,1]$ and by the Tietze theorem this mapping
can be extended to the entire range of $f$.
The composition would give a continuous mapping from $K$ onto $[0,1]$, contradicting (4).

(5) $\Rightarrow$ (6) clear.

(6) $\Rightarrow$ (1) If $K$ had a subset with no isolated points, its closure $L$
would have no isolated points as well. By the Tietze theorem it is enough to construct a continuous
map from $L$ onto $[0,1]$. As every open subset of $L$ is infinite, one can
construct a Cantor tree of open sets with inclusion of closures, i.e.,
open $U_s$ for $s\in \{0,1\}^{<\N}$ such that $\overline{U_{s^\frown1}},
\overline{U_{s^\frown0}}\subseteq U_s$
and $\overline{U_{s^\frown1}}\cap
\overline{U_{s^\frown0}}=\emptyset$ for each $s\in \{0,1\}^{<\N}$.
Then one can define a continuous function $\phi$ from $M=\bigcap_{n\in \N}\bigcup_{|s|=n}\overline U_s$
onto $\{0,1\}^\N$ by
puting $\phi(x)=y$ whenever $x\in\overline U_s$ if and only if $y|n=s$ for $n=|s|$.
Then  compose $\phi$ with a  function from
$\{0,1\}^\N$ onto $[0,1]$. Finally using the Tietze Theorem one can extend
the composition from $M$ to $L$ and then to $K$.

\end{proof}

For completeness we also add the following theorem (cf. Theorem \ref{ecStar-banach}) which shows the
equivalent properties for $K$ being compact scattered Hausdorff 
in terms of the Banach space of continuous functions, its Banach dual or bidual:

\begin{theorem}[\cite{pelczynski-semadeni}]\label{ec-banach} Suppose that $K$ is a compact Hausdorff space.
All the following conditions are equivalent to $K$ being scattered:
\begin{enumerate}

\item Every Radon measure on $K$ is atomic, i.e., of the form
$\mu=\Sigma_{n\in \N} t_n\delta_{x_n}$ where $x_n\in K$ and $\Sigma_{n\in \N}|t_n|<\infty$.
\item The dual of $C(K)$, the space of Radon measures on $K$, is 
isometric to $\ell_1(K)$.
\item The bidual of $C(K)$ is $\ell_\infty(K)$ where the inclusion $C(K)\subseteq \ell_\infty(K)$
is the canonical embedding of $C(K)$ into the bidual $C(K)^{**}$.
\item The Banach dual of every separable Banach subspace of $C(K)$ is separable.
\end{enumerate}
\end{theorem}

There are at least two other long list of equivalent conditions for $K$ to be
scattered in terms of the Radon-Nikodym property in the vector valued integration
(\cite{vector-measures}) or in terms of strong differentiability in Banach spaces (\cite{dgz}).

\section{Atoms and the ideals they generate in noncommutative topology}

\subsection{Minimal projections}

\begin{lemma}\label{minimal-one-dimensional} Suppose
that $\mathcal H$ is a Hilbert space and $K$ is a  compact Hausdorff space. Minimal
projections in $\mathcal B(\mathcal H)$ are orthogonal projections
onto one-dimensional subspaces. Minimal projections of $C(K)$ are
characteristic functions $\chi_{\{x\}}$ of isolated points $x\in K$.
\end{lemma}

 One easily checks
that if $p\in \mathcal A$ is a minimal projection, then there is a linear functional
$\lambda_p$ such that $pap=\lambda_p(a)p$ for each $a\in \mathcal A$.
For  $p=\chi_{\{x\}}$ and $\mathcal A=C(K)$ the functional is $\delta_x$, the evaluation at $x$.
It can be  proved that for any $C^*$-algebra $\lambda_p$
is a pure state (see Chapter 5 of \cite{murphy}).
Recall that pure states are extreme points of the set of all positive linear functionals of norm $\leq 1$.
In the case of $C(K)$ they are exactly $\delta_x$ for all $x\in K$.

\begin{lemma}[1.4.1 of \cite{invitation}]\label{minimal-in-compact}
 Suppose that $\mathcal A\subseteq \mathcal B(\mathcal H)$ is
a $C^*$-algebra. An element $P\in \mathcal K(\mathcal H)$ is a minimal
projection in $\mathcal A$ if and only if $P$ is a projection  which is minimal
with respect to the usual ordering of projections, i.e., 
there is no nonzero projection $Q\in \mathcal A$ such that $Q\leq P$ and $Q\not=P$.
\end{lemma}

In general the projections which are minimal with respect to the ordering do not need
to be minimal in the sense of Definition \ref{definition-minimal}. For
example the unit of $C([0,1])$ is minimal with respect to the order but
does not satisfy Definition \ref{definition-minimal}. This example also shows that
a minimal projection in a subalgebra may not be minimal in a bigger algebra. 
However we have the following:

\begin{lemma}\label{minimal-in-ideal} Suppose that $\A$ is a $C^*$-algebra, $\I\subseteq\A$ is its ideal and
$\mathcal B\subseteq \A$ is its masa. Then a minimal projection in $\I$ or in
$\B$ is a minimal projection in $\A$.
\end{lemma}
\begin{proof}Suppose $p$ is a minimal projection in $\mathcal I$.
We have a functional $\phi\in \mathcal I^*$ such that $pcp=\phi(c)p$ for every $c\in \mathcal I$.
Consider any $a\in \mathcal A$. Then $pap\in \mathcal I$, so $pap=p(pap)p=\phi(pap)p$
as required for  a minimal projection in $\mathcal A$.

Let $\B$ be a masa of $\A$. The Gelfand transform
$a \rightarrow \hat{a}$ is an $*$-isomorphism between $\mathcal B$ and $C_0 (X)$  for some locally
compact Hausdorff space $X$.  Let the
characteristic function $\chi_{\{x\}}$ of $\{x\}$ be a minimal projection in $C_0 (X)$ and
 $p$ be the unique minimal projection of $\mathcal B$ such that
$\hat{p}= \chi_{\{x\}}$.  We will show that $p$ is also a minimal projection in $\mathcal A$.
Take arbitrary elements $a\in \mathcal A$ and $b\in \mathcal B$. There is $\lambda\in \mathbb C$
such that $pbp= \lambda p$ and
$$
papb= papbp= \lambda pap = pbpap = bpap.
$$
Therefore $pap$ commutes with $\mathcal B$. By the maximality of $\mathcal B$, $pap$ belongs
to $\mathcal B$ and hence $p^2ap^2=pap= \gamma p$ for some scalar $\gamma$, since
$p$ is a minimal projection in $\mathcal B$.
\end{proof}

\begin{lemma}\label{atoms-unitization}
For any $C^*$-algebra $\mathcal A$, $\I^{At}(\widetilde{\A})= \I^{At}(\mathcal A)$.
\end{lemma}
\begin{proof}
The  inclusion $\I^{At}(\A)\subseteq \I^{At}(\widetilde{\A})$ follows from the fact that $\A$ is an ideal of 
$\widetilde{\A}$ and Lemma \ref{minimal-in-ideal}.

For the proof of the other inclusion we may assume that $\mathcal A$ is not unital.
Suppose that $(a, \alpha)\in \widetilde{\A}$ is  a minimal projection for
$\alpha\not=0$. From this hypothesis we will derive a contradiction with $\mathcal A$ being
nonunital.
The fact that $(a,\alpha)$ must be an idempotent implies that $\alpha=1$ and $a^2=-a$,
while the fact that $(a, 1)$ must be self-adjoint implies that $a$ is selfadjoint.  
The minimality of $(a, 1)$ implies that $(a, 1)(b, 0)(a, 1)=\lambda(a, 1)$
for some $\lambda\in \C$ which must be $0$. So
$aba+ba+ab+b=0$ for every $b\in \mathcal A$. 

Let  a net $(e_\xi)_{\xi\in \Xi}$ be an approximate unit for $\A$ (see Theorem 1.8.2 of \cite{invitation}),
that is in particular we have $\|xe_\xi-x\|\rightarrow 0$ and $\|e_\xi x-x\|\rightarrow 0$ for
every $x\in \A$. Considering $b=e_\xi$ we get that $ae_\xi a+e_\xi a +ae_\xi=-e_\xi$
for all $\xi\in \Xi$ but the left hand side converges to $a^2+2a=a$, so the right hand side $(-e_\xi)$ must
converge to $a$ as well. This shows that $-a$ is the unit for $\A$, a contradiction.

\end{proof}

We conclude this section by the next two elementary lemmas, exhibiting how the existence of certain projections
in a $C^*$-algebra (or its subalgebras) can prevent it from being scattered, by imposing the existence of a copy of 
$C([0,1])$ 
in it.

\begin{lemma}\label{none-interval} Suppose that $\mathcal A$ is a $C^*$-algebra
and $p\in \mathcal A$ is a projection which is not a minimal projection
but such that there is no nonzero
projection $q\lneqq p$. Then $p\mathcal Ap$ contains a copy of $C([0,1])$.
\end{lemma}
\begin{proof} Pick $q=prp$ for some $r\in \mathcal A$ which is
not of the form $\lambda p$ for some $\lambda\in \C$.
We have that $r=r_1+ir_2$
where $r_1=(r+r^*)/2$ and $r_2=(r-r^*)/2i$ are self adjoint, so we may assume that $r$ is
self adjoint.
As $prp=p^2rp=prp^2$, we have that $p$ and $q$ commute.
  Since  $r$ is self adjoint, we have  $q^*=(prp)^*=p^*r^*p^*=q$. So
$p$ and $q$ are self adjoint commuting elements of $p\mathcal Ap$,
so the subalgebra
 $\mathcal B\subseteq p\mathcal Ap$ which they generate is abelian and unital
and  so isomorphic to $C(K)$ for some  compact $K$.
Projection $p$ in such an algebra corresponds to the characteristic function of $K$.
Note that $K$ has more than one element, since otherwise  $q$ would be of the form $\lambda p$. 
As there is no projection smaller than $p$,
it follows that $K$ is connected, so
there is a continuous surjection $\phi: K\rightarrow [0,1]$.  
Now $\iota: C([0,1])\rightarrow \B\subseteq p\A p$ given by $\iota(f)=\phi\circ f$
is the required $*$-embedding of $C([0,1])$ into $p\A p$.
\end{proof}

\begin{lemma}\label{many-interval} Suppose that
  $\mathcal A$ is a $C^*$-algebra with a nonzero projection such that every 
nonzero projection has a strictly smaller nonzero projection.
Then $\mathcal A$ contains a copy of $C([0,1])$.
\end{lemma}
\begin{proof}
Using the hypothesis we can construct  projections $(p_s: s\in \{0,1\}^{<\N})$
such that $p_{s^\frown 0}+p_{s^\frown 1}=p_s$ for each $s\in \{0,1\}^{<\N}$
and all the projections are nonzero. If $s\subseteq t$ then $p_t\leq p_s$
and so $p_tp_s=p_tp_s=p_t$, that is all $p_s$s commute and
so the algebra generated by them is $*$-isomorphic to a $C(K)$
for some compact Hausdorff $K$. The projections $p_s$ correspond
in $C(K)$ to a tree of nonempty clopen sets $(U_s: s\in \{0,1\}^{<\N})$
such that $U_{s^\frown 0}+U_{s^\frown 1}=U_s$ for each $s\in \{0,1\}^{<\N}$.
Like in the proof of \ref{eq-compact-scattered} (5) $\Rightarrow$ (1)
this yields a surjective continuous map
$\phi: K\rightarrow [0,1]$ and so $C([0,1])\subseteq C(K)\subseteq \mathcal A$,
as required.
\end{proof}

\subsection{Irreducible representations and the ideal $\I^{At}(\A)$}
An irreducible
representation $\pi$ of a $C^*$-algebra $\mathcal A$ on a Hilbert space $\mathcal H$ 
is a $^*$-homomorphism  from $\mathcal A$ into $\mathcal B(\mathcal H)$ such
that its range $\pi[\A]$ is an irreducible subalgebra of $\mathcal B(\mathcal H)$, i.e., it has no nontrivial 
$\pi[\A]$-invariant closed
subspace in $\mathcal H$. Throughout this section we use different equivalent conditions 
for a representation to be
irreducible (see e.g., \cite[II.6.1.8.]{Blackadar}).
In the commutative case, irreducible representations correspond to the multiplicative functionals,
i.e., functionals of the form $\pi_x: C_0(X) \rightarrow \mathbb C$,
where $\pi_x (f)=f(x)$ for some $x\in X$. Here
the Hilbert space $\HH$ is the one-dimensional space $\C$.

\begin{lemma}\label{irreducible-subrepresentations}
Suppose that $\mathcal A$ is a $C^*$-algebra and $\pi:\mathcal A \rightarrow \mathcal B(\mathcal H)$ is
a representation of $\mathcal A$.
Let $p$ be a minimal projection in $\mathcal A$, a unit vector $v$
be in the range of $\pi(p)$ and $\mathcal H_0 = [\pi[\mathcal A] v]$.
Then $\pi[\mathcal A]|_{\mathcal H_{0}}$ is an irreducible subalgebra
 of $\mathcal B (\mathcal H_{0})$.
%Moreover if $\mathcal A$
% is generated by minimal projections then  $\pi[\mathcal A]|_{\mathcal H_{0}}
%=\mathcal K (\mathcal H_{0})$.
\end{lemma}

\begin{proof}
Since $\HH_0$ is $\pi[\mathcal A]$-invariant,
the map $\pi_{0}: \mathcal A \rightarrow \mathcal B (\mathcal H_{0})$
defined by $\pi_0(a)=\pi(a)|\mathcal H_{0}$
 is a subrepresentation of $\pi$. Assume that $T\in \mathcal B (\mathcal H_0)$
 commutes with $\pi_{0} [\mathcal A]$. We will show that $T\in \mathbb C I$
which is sufficient as the triviality of the commutant is equivalent to being irreducible (\cite[II.6.1.8.]{Blackadar}).
 By replacing $T$ with $T - (T v, v)I$, we may assume $(T v,v)=0$. Under this
assumption we will
show that $T=0$. Now if $a, b \in \mathcal A$ we
have $\pi_{0}(pb^{*}ap)=\lambda \pi_0(p)$ for some $\lambda\in \C$. Therefore
$$
(T\pi_{0} (a) v, \pi_{0}(b)v)=(T\pi_{0} (a)\pi_0(p) v, \pi_{0}(b)\pi_0(p)v)= 
(T \pi_{0}(pb^{*}ap)v, v)=$$
$$=\lambda (T \pi_{0}(p)v, v)=\lambda (T v , v)=0.
$$
Since $\mathcal H_0 = [\pi[\mathcal A] v]$,  arbitrary two vectors of $\mathcal H_0$
can be approximated by
vectors  of the form $\pi_{0}(a) v$ and $\pi_{0}(b)v$ for some $a, b\in \mathcal A$,
so it follows that $T=0$.

%The last statement follows immediately from
% Lemma  \ref{minimal-one-dimensional} and Lemma \ref{preserving-minimal}.
\end{proof}

\begin{lemma}\label{preserving-minimal} Suppose that $\mathcal A$ is an irreducible
 $*$-subalgebra of $\B(\mathcal H)$.  If $P$ is a minimal projection in $\mathcal A$, then
$P$ is a minimal projection in $\mathcal B(\mathcal H)$
\end{lemma}
\begin{proof}  By Lemma \ref{minimal-one-dimensional} it is enough
to show that $P$ can not be an orthogonal projection onto a subspace with 
dimension more than one. If this was the case, take a nonzero vector $v$ in
the range of $P$ and consider $\mathcal X=[\mathcal A v]\subseteq \mathcal H$.
We claim that under this hypothesis $\mathcal X$ would be a proper $\A$-invariant subspace of $\mathcal H$
contradicting the irreducibility of $\A$.
It is clear that it is an $\A$-invariant subspace of $\mathcal H$.
As $P(v)=v$ and for some $\lambda\in \C$ by the minimality
of $P$  we have  $PTP=\lambda P$ , we obtain $PT(v)=(PTP)(v)=\lambda P(v)\in \C v$
for every $T\in \mathcal A$. But if the range of $P$ is
more than one-dimensional, there is a nonzero vector $w$ in the range of
$P$  which is orthogonal to $v$. Since $P(w)=w$,  for each $T\in \A$ we have 
$$(T(v), w)=((I-P)Tv, w)+(PTv, w)=(Tv, (I-P)w)=(Tv, 0)=0$$
which shows that $\mathcal X\subseteq \{w\}^\perp$ is a proper $\mathcal A$-invariant  subspace of $\mathcal H$,
which contradicts the irreducibility of $\mathcal A$.

\end{proof}

\begin{lemma}\label{minimal-only-one} Suppose that $\mathcal A$ is a $C^*$-algebra and $p$ is its
minimal projection. Let $\pi_i: \mathcal A\rightarrow \mathcal B(\mathcal H_i)$
for $i=1,2$ be irreducible representations of $\mathcal A$. If
$\pi_1(p)\not=0\not=\pi_2(p)$, then $\pi_1$ and $\pi_2$ are
unitarily equivalent.
\end{lemma}
\begin{proof}
Let $\phi\in \mathcal A^*$ be a functional such that $pap=\phi(a)p$ for every
$a\in \mathcal A$.  As $\pi_1(p)\not=0\not=\pi_2(p)$ we can 
pick norm one vectors $v$ and $w$ in the ranges of $\pi_1(p)$ and $\pi_2(p)$, respectively.
% By Lemma \ref{preserving-minimal}
% $\pi_1(p)$ and $\pi_2(p)$
%are one-dimensional projections in $\mathcal H_1$ and $\mathcal H_2$,
%respectively.
Then for every $a\in \mathcal{A}$ we have
\begin{eqnarray}
\nonumber       \| \pi_1 (a) v  \|^2 & = &  \| \pi_1 (ap) v  \|^2= (\pi_1 (ap) v,\pi_1 (ap) v)  \\
\nonumber                                         & = &   ( \pi_1 (pa^{*}ap) v, v) =\phi(a^*a) (v, v)     \\
\nonumber                                         & = &   \phi(a^*a) (w,w)= (\pi_{2} (pa^{*} a p) w, w)  \\
\nonumber                                         & = & \|\pi_2 (a) w\|^2.
\end{eqnarray}
This implies that  that $\pi_1(a)v=\pi_1(a')v$ is equivalent to $\pi_2(a)w=\pi_2(a')w$ for any $a, a'\in \A$.
Therefore  the map $U$ sending $\pi_1(a)v$ to $\pi_2(a)w$ extends to a well-defined linear  isometry
from   $[\pi_1[\mathcal A]v]$ onto $[\pi_2[\mathcal A]w]$.

By the irreducibility
 of $\pi_1$ and $\pi_2$ both $v$ and $w$
are cyclic vectors for $\mathcal H_1 $ and $\mathcal H_2$, 
respectively, i.e., $[\pi_1[\mathcal A]v]=\HH_1$ and $[\pi_2[\mathcal A]w]=\HH_2$.
 Therefore $U$ is a unitary from $\HH_1$ onto $\HH_2$.
 Also if  $u\in \mathcal H_1$ is
such that  there is $ b\in \mathcal A$ satisfying $\pi_1(b)v=u$ for every  $a\in \A$ we have
$$(U^{-1}\pi_2(a)U)(u)=(U^{-1}\pi_2(a)U)(\pi_1(b)v)=  U^{-1}\pi_2(a)\pi_2(b)w=$$
$$=U^{-1}\pi_2(ab)w=\pi_1(ab)v=\pi_1(a)u.$$
As the set of $u$s as above is dense in $\HH_1$ we conclude that  $\pi_1$ and $\pi_2$ are unitarily equivalent.
\end{proof}

\begin{definition} Let  $\mathcal A$ be a $C^*$-algebra and $(\pi_i, \mathcal H_i)_{i\in I}$
a maximal collection of pairwise unitarily inequivalent irreducible representations.
 Then $\pi:\mathcal A\rightarrow \prod_{i\in I}\mathcal B(\mathcal H_i)$
given by
$$\pi(a)|\mathcal H_i=\pi_i(a)$$
is called the reduced atomic representation of $\mathcal A$.
\end{definition}

\begin{lemma}\label{reduced-faithful} The reduced atomic representation of a $C^*$-algebra
is faithful.
\end{lemma}
\begin{proof} Suppose that $(\pi_i, \mathcal H_i)_{i\in I}$
a maximal collection of pairwise unitarily inequivalent irreducible representations,
 but there is $a\in \mathcal A$ such that $\pi_i(a)=0$
for all $i\in I$ and $a\not=0$.  There is an irreducible representation $\sigma$ such that
$\sigma(a)\not=0$ by Theorem 5.1.12. of \cite{murphy}. Of course such a
representation can not be unitarily equivalent to any  $\pi_i$s, which contradicts
the maximality of $(\pi_i)_{i\in I}$.
\end{proof}

\begin{proposition}\label{only-atoms-to-compact}
Suppose that $\mathcal A$ is a $C^*$-algebra, $\mathcal I$
is its ideal
and $(\pi, \mathcal H)$ is any faithful
representation of  $\mathcal I$.  Then
$$\pi[\mathcal I]\cap \K(\mathcal H)\subseteq\pi[\I^{At}(\mathcal A)].$$
\end{proposition}
\begin{proof}
Let $T\in \pi[\mathcal I]\cap \mathcal K(\mathcal H)$ be self-adjoint. Then by the
spectral theorem $T=\Sigma_{n\in \N}\lambda_n P_n$ where $P_n$s are finite dimensional projections.
Moreover, spectral theorem implies that $P_n\in \pi[\mathcal I]$. 
Each finite dimensional $C^*$-algebra is generated by its minimal projections (by the
representation of finite dimensional $C^*$-algebras as finite products of matrix algebras).
This implies that
$T$ is in $\I^{At}(\pi[\mathcal I])=\pi[\I^{At}(\mathcal I)]$, the last equality follows from the faithfulness of
$\pi$. However, $\I^{At}(\mathcal I)\subseteq \I^{At}(\mathcal A)$, by
Lemma \ref{minimal-in-ideal} and therefore $\pi[\mathcal I]\cap \K(\mathcal H)\subseteq\pi[\I^{At}(\mathcal A)]$.
\end{proof}

\begin{proposition}\label{atoms-in-reduced-atomic} Suppose that $(\pi, \mathcal H)$ is the reduced atomic
representation of a $C^*$-algebra $\mathcal A$.  Then
$$\pi[\mathcal A]\cap \K(\mathcal H)=\pi[\I^{At}(\mathcal A)].$$
\end{proposition}
\begin{proof}
Let $(\pi_i, \mathcal H_i)_{i\in I}$ be the maximal pairwise inequivalent irreducible representations of
$\mathcal A$ which give rise to $\pi$. By Lemmas \ref{minimal-only-one}
for every minimal projection $p\in \mathcal A$,
$\pi(p)$ is $0$ on all but one $\mathcal H_i$. By Lemma 
\ref{preserving-minimal} and Lemma \ref{minimal-one-dimensional}  $\pi(p)$ is a one-dimensional
projection on this  $\mathcal H_i$. It follows that $\pi[\I^{At}(\mathcal A)]\subseteq \K(\mathcal H)$.

The other inclusion follows from Proposition \ref{only-atoms-to-compact} applied for
$\mathcal I=\mathcal A$.
\end{proof}

Recall that $CCR(\A)$ is the ideal of all elements of a $C^*$-algebra $\A$ which
are sent onto a compact operator by each irreducible representation of $\A$ (see Section 1.5 of 
\cite{invitation}).

\begin{proposition}\label{ccr} Suppose that $\A$ is a $C^*$-algebra. Then $\I^{At}(\A)\subseteq
CCR(\A)$. 
\end{proposition}
\begin{proof} 
Suppose that $\pi:\A\rightarrow \B(\HH)$ is an irreducible representation of $\A$ on some Hilbert space
$\HH$. If $p\in \A$ is a minimal projection in $\A$, then by Lemma 
\ref{preserving-minimal} and Lemma \ref{minimal-one-dimensional}  $\pi(p)$ is a one-dimensional
projection on this  $\mathcal H$, i.e., $\pi(p)\in \K(\HH)$. As 
the minimal projections  of $\A$ generate $\I^{At}(\A)$ we conclude that $\pi[\I^{At}(\A)]\subseteq \K(\HH)$
as required.
\end{proof}

\begin{proposition}\label{atoms-are-compact}
Suppose $\mathcal A$ is a $C^*$-algebra.
Then $\I^{At}(\mathcal A)$ is isomorphic to a subalgebra of $\mathcal K(\mathcal H)$
 for some Hilbert space $\mathcal H$.
\end{proposition}

\begin{proof}
Apply Proposition \ref{atoms-in-reduced-atomic} and Lemma \ref{reduced-faithful}.
\end{proof}

\begin{proposition}\label{atoms-are-max-ideal} Suppose $\mathcal A$ is a $C^*$-algebra.
Then $\I^{At}(\mathcal A)$ is an ideal of $\mathcal A$.
It contains all ideals isomorphic to a subalgebra of compact operators on a Hilbert space.
Therefore there is no bigger ideal which can be faithfully mapped into an
algebra of compact operators on a Hilbert space.
\end{proposition}
\begin{proof}

Let $\pi$ be the reduced atomic representation of $\A$ on $\HH$. By Proposition \ref{atoms-in-reduced-atomic}
we have
$\pi[\mathcal A]\cap \K(\mathcal H) = \pi[\I^{At}(\mathcal A)].$
Since $\K(\mathcal H)$ is an ideal of $\B(\mathcal H)$,  $\pi[\I^{At}(\mathcal A)]$
is an ideal of $\pi[\mathcal A]$. Therefore  $\I^{At}(\mathcal A)$ is an ideal
of $\mathcal A$. 
Any ideal $\I$ isomorphic to a subalgebra of $\K(\mathcal H)$ is generated as an algebra by its minimal
projections (by inspection in subalgebras of $\K(\mathcal H)$).  
The minimal projections in $\mathcal I$ are minimal projections in $\mathcal A$ (Lemma \ref{minimal-in-ideal}). Thus
$\mathcal I\subseteq \I^{At}(\mathcal A)$.
\end{proof}

\begin{proposition}\label{intersection-of-atomic} Suppose
that $\mathcal A$ is a $C^*$-algebra and $\mathcal B$ is a subalgebra
of $\mathcal A$. Then $\I^{At}(\mathcal A)\cap \mathcal B\subseteq \I^{At}(\mathcal B)$. If $\mathcal J$ is an ideal 
of $\mathcal A$, then $\I^{At}(\mathcal J) = \I^{At}(\mathcal A) \cap \mathcal J$.
\end{proposition}
\begin{proof}
By  Proposition \ref{atoms-are-max-ideal},
 $\I^{At}(\mathcal A)\cap \mathcal B$ is an ideal of $\mathcal B$.
Since $\I^{At}(\mathcal A)$ can be mapped faithfully into an
algebra of compact operators (Proposition \ref{atoms-are-compact}), so can
$\I^{At}(\mathcal A)\cap \mathcal B$. However by Proposition \ref{atoms-are-max-ideal}
any such ideal of $\mathcal B$ is included in $\I^{At}(\mathcal B)$. The second statement immediately  follows from 
the first statement and Lemma \ref{minimal-in-ideal}.
\end{proof}

\subsection{Atoms and essential ideals}

Recall that an ideal $\mathcal I$ in a $C^*$-algebra $\mathcal A$
is called essential if and only if $\mathcal I\cap\mathcal J\not=\{0\}$
for any  nonzero ideal of $\mathcal A$. Also $\mathcal I$  is essential if and only if 
$\mathcal I^\perp=\{a\in \mathcal A: a\I=0\}$
is the zero ideal (see II.5.4.7 of \cite{Blackadar}).

\begin{lemma}\label{ideal-in-ideal} Suppose that $\A$ is a $C^*$-algebra, $\I\subseteq \A$ is its essential ideal
and $\J\subseteq \I$ is an essential ideal of $\I$. Then $\J$ is an essential ideal of $\A$.
\end{lemma}
\begin{proof} To see that $\J$ is an ideal of $\A$, take $j\in \mathcal J$ and $a\in \A$.
Let $(j_\xi)_{\xi\in \Xi}$ be an approximative unit for $\J$, in particular $\|jj_\xi -j\|$
converges to $0$. We have $aj\in \I$ as $\I$ is an ideal of $\A$ and since
$(ajj_\xi)_{\xi\in \Xi}$ converges to $aj$ and each $ajj_\xi$ is in $\J$, it follows that 
$aj\in \J$ as required. The proof for
$ja$ is analogous.

For the essentiality of $\J$ suppose that there is $a\in \A$ such that $a\J=0$.
However, the essentiality of $\I$ in $\A$ implies that there is $i\in \I$ such that
$ia\not=0$. Now the essentiality of $\J$ in $\I$ yields $j\in \J$ such that
$iaj\not=0$. It follows that $aj\not=0$, which is a contradiction.
\end{proof}

\begin{proposition}\label{embedding-scattered} Suppose that $\mathcal A$ is a $C^*$-algebra 
and $\mathcal I\subseteq \mathcal A$ is its essential ideal such that there is an 
embedding
 $i: \mathcal I\rightarrow \mathcal K(\mathcal H)$ for some Hilbert space $\mathcal H$.
Then there is an extension of this embedding to an
 embedding $\tilde{i}: \mathcal A\rightarrow \mathcal B(\mathcal H)$.
\end{proposition}
\begin{proof}
Let $\mathcal M(\mathcal I)$ be the multiplier algebra of $\mathcal I$,
that is, a unital $C^*$-algebra containing $\mathcal I$ as an essential ideal 
such that it is universal in the sense that
 for any $C^*$-algebra $\mathcal B$ containing $\mathcal I$ as an ideal
there is a unique homomorphism $h: \mathcal B\rightarrow \mathcal M(\mathcal I)$ extending
the identity map on $\mathcal I$ (see \cite[I.7.3.1]{Blackadar}). 
Since $i$ is a faithful representation of 
$\I$, we can identify $\mathcal  M(\I)$ as the idealizer of the image of $\I$ inside $\mathcal B(\HH)$
 (see \cite[II.7.3.5]{Blackadar}).
Let $\tilde{i}: \mathcal A\rightarrow \mathcal M(\I)\subseteq \mathcal B(\HH)$ be the homomorphism 
obtained from
the universality of the multiplier algebra, which extends $i$.   We have that
$\mathcal I^\perp=\{0\}$. Thus for every nonzero $a\in \mathcal A$
there is $b\in \mathcal I$ such that $ab\not=0$.
However $ab$ belongs to $\mathcal I$ since it is an ideal,
so $\tilde i (ab)= i (ab)\not=0$  which means that $\tilde i (a)\not=0$.
Therefore  $\tilde i$ is an embedding into $\mathcal B(\mathcal H)$.
\end{proof}

\begin{proposition}\label{essential-maximal} Suppose that $\mathcal A$ a  $C^*$-algebra
and $\mathcal I\subseteq \I^{At}(\mathcal A)$ is its essential ideal. Then $\mathcal I=\I^{At}(\mathcal A)$.
\end{proposition}
\begin{proof}
We know that $\I^{At}(\mathcal A)$ is isomorphic to a subalgebra of $\mathcal K(\mathcal H)$
for some Hilbert space $\mathcal H$ (Proposition \ref{atoms-are-compact}). 
By Proposition \ref{embedding-scattered} we may assume that
$\mathcal A\subseteq \mathcal B(\mathcal H)$ and $\mathcal I\subseteq \mathcal K(\mathcal H)$.
 Therefore $\I^{At}(\A)$ is equal to $\bigoplus_{i\in I} \mathcal K(\mathcal H_i)$ 
for a family of Hilbert spaces $\HH_i\subseteq \HH$ for $i\in I$ (\cite[Theorem 1.4.5.]{invitation}). 
Hence $\mathcal I$ is
an essential ideal of $\bigoplus_{i\in I} \mathcal K(\mathcal H_i)$, which means that
it must contain nonzero elements of each $\mathcal K(\mathcal H_i)$. 
By the fact that each $\K(\HH_i)$ is simple, i.e., has no nontrivial ideals, it follows that
$\mathcal I$ is the entire $\bigoplus_{i\in I} \mathcal K(\mathcal H_i)$.
That is $\mathcal I=\I^{At}(\mathcal A)$, as required.
\end{proof}

\begin{proposition}\label{essential-compact} Suppose that $\mathcal I$ is an essential
ideal of a $C^*$-algebra $\mathcal A$ which  is isomorphic
to a subalgebra of $\mathcal K(\mathcal H)$ for some Hilbert space $\mathcal H$.
Then $\mathcal I=\I^{At}(\mathcal A)$.
\end{proposition}
\begin{proof} 
Since $\I^{At}(\A)$ is the largest ideal of $\A$ which is 
 isomorphic
to a subalgebra of $\mathcal K(\mathcal H)$ for some Hilbert space $\mathcal H$
 (Proposition \ref{atoms-are-max-ideal}), $\I$ is contained in $\I^{At}(\A)$. Therefore 
 Proposition \ref{essential-maximal} implies that $\mathcal I=\I^{At}(\mathcal A)$.

\end{proof}

\vfill
\break
\section{The noncommutative case - scattered $C^*$-algebras}

\begin{reptheorem}[\ref{theorem1}]{\rm\cite{jensen1, jensen2, wojtaszczyk}}
 Suppose that $\mathcal A$ is a $C^*$-algebra.
The following conditions are equivalent:
\begin{enumerate}

\item\label{ec-homomorphic} Every non-zero $*$-homomorphic image of $\mathcal A$ has a minimal projection.
\item\label{ec-sequence-new} There is an ordinal $ht(\mathcal A)$
and a continuous  increasing sequence of closed ideals $(\I_\alpha)_{\alpha\leq ht(\mathcal A)}$
such that  $\I_0=\{0\}$, $\I_{ht(\mathcal A)}= \mathcal A$ and
$$\I^{At}(\mathcal A /\I_\alpha)=\{[a]_{\I_\alpha}: a\in \I_{\alpha+1}\},$$
for every $\alpha < ht(\mathcal A)$.
\item\label{ec-sequence-old} There is an ordinal $m(\A)$
and a continuous  increasing sequence of ideals $(\J_\alpha)_{\alpha \leq m(A)}$
such that $\J_0=\{0\}$, $\J_{m(\A)} = \mathcal A$ and  $\J_{\alpha+1}/\J_\alpha$ is 
an elementary $C^*$-algebra (i.e., $^*$-isomorphic to the algebra
of all compact operators  on a Hilbert space)
for every $\alpha < m(\A)$.
\item\label{ec-subalgebra} Every non-zero subalgebra of $\mathcal A$ has a minimal projection.
\item\label{ec-no-interval} $\A$ does not contain a copy of the $C^*$-algebra $C_0((0,1])=\{f\in C((0,1]):
\lim_{x\rightarrow 0}f(x)=0\}$.
\item\label{ec-spectrum} The spectrum of every self-adjoint element is countable.
\end{enumerate}
\end{reptheorem}
\begin{proof}
%In the proof we can assume that $\mathcal A$ is unital since it is easy to check that $\mathcal A$
%satisfies any of the above conditions if and only if $\widetilde{\A}$, the unitization of $\A$, does.

(\ref{ec-homomorphic}) $\Rightarrow$ (\ref{ec-sequence-new}) 
We define the sequence $(\I_{\alpha})_{\alpha<ht(\A)}$ by induction. Put $\I_0 = \{0\}$.
At each successor stage $\alpha +1$, if $\mathcal A / \I_{\alpha}$ is non-zero,
then $\I^{At}(\mathcal A / \I_{\alpha})$ is non-zero by (\ref{ec-homomorphic}). Let 
$\I_{\alpha +1}= \sigma_{\alpha}^{-1}(\I^{At}(\mathcal A / \I_{\alpha}))$, 
where $\sigma_\alpha : \mathcal A \rightarrow \mathcal A /\I_\alpha$ is the  quotient map. 
If $\gamma$ is a limit ordinal let $\I_\alpha = \overline{\bigcup_{\alpha < \gamma} \I_\alpha}$. 
The induction
has to terminate at some ordinal $ht(\mathcal A)$, whose
 cardinality does not exceed ${|\mathcal A|}$.

(\ref{ec-sequence-new}) $\Rightarrow$ (\ref{ec-sequence-old}) For each ordinal $\alpha < ht(\mathcal A)$ 
we have $\I_{\alpha+1}/ \I_\alpha= \I^{At}(\mathcal A / \I_\alpha)$ and hence 
by Proposition \ref{atoms-are-compact} the quotient $\I_{\alpha+1}/ \I_\alpha$ is isomorphic to 
a subalgebra of the algebra of compact operators. Such algebras are isomorphic to algebras
of the form 
$\bigoplus_{i\in \Lambda_{\alpha}} \mathcal K(\mathcal H_{\alpha,i})$ for some Hilbert spaces 
$\mathcal H_{\alpha, i }$ and a totally ordered set $\Lambda_{\alpha}$ (1.4.5. of \cite{invitation}). 
Let $\nu_\alpha : \I_{\alpha+1}/ \I_\alpha \rightarrow \bigoplus_{i\in \Lambda_{\alpha}} 
\mathcal K(\mathcal H_{\alpha,i})$ be such an isomorphism and define 
$\J_{\alpha, i}= \sigma_{\alpha}^{-1} \circ \nu_{\alpha}^{-1} (\bigoplus_{j\leq i} 
\mathcal K(\mathcal H_{\alpha,j}))$. Order 
$\Gamma = \{(\alpha,i) : i\in \Lambda_{\alpha}, \alpha \leq ht(\mathcal A)\}$
 lexicographically and let $m(\A)$ be the order type of this set. 
Re-enumerate the set of $\{\J_{\alpha,i} : (\alpha, i) \in \Gamma \}$ by
 the ordinal $m(\A)$ as $\{\J_\gamma : \gamma \leq m(\A)\}$. 
It is easy to check that $(\J_{\gamma})_{\gamma \leq m(\A)}$ has the required properties.

(\ref{ec-sequence-old}) $\Rightarrow$ (\ref{ec-subalgebra}) 
Let $(\J_\alpha)_{\alpha \leq  m(\A)}$
be a composition series as in (\ref{ec-sequence-old}).

First assume that $\mathcal C$ is a unital subalgebra of $\mathcal A$.  By passing to a masa of $\mathcal C$
and using Lemma
\ref{minimal-in-ideal}, we may assume that $\mathcal C$ is commutative
and hence isomorphic to an algebra of the form $C(K)$ for some compact Hausdorff $K$.
Consider the sequence $(\J_\alpha')_{\alpha\leq m(\A)}$ of ideals of $C(K)$ 
defined by $\J_\alpha'=\J_\alpha\cap \mathcal C$. 
Then $\J_{\alpha+1}'/\J_{\alpha}'$ is isomorphic
to a subalgebra of $\J_{\alpha+1}/\J_{\alpha}$, but being a commutative
subalgebra of the algebra of compact operators $\J_{\alpha+1}'/\J_{\alpha}'$ is isomorphic to 
$c_0(\Gamma_\alpha)$
for some discrete space $\Gamma_\alpha$. As all ideals in $C(K)$ are of the form
$\{f\in C(K): f|F=0\}$ for some closed $F\subseteq K$, we obtain a 
corresponding decreasing
continuous sequence $(F_\alpha)_{\alpha \leq m(\A)}$. Since $C(K)/\{f\in C(K): f|F=0\}$
is canonically isomorphic to $C(F)$, we conclude that 
$\{f\in C(F_{\alpha}): f|F_{\alpha+1}=0\}$ is isomorphic to $c_0(\Gamma_\alpha)$, and therefore
 $F_\alpha\setminus F_{\alpha+1}$ is discrete, i.e. consists only
of points isolated in $F_\alpha$. Therefore, as in the proof of (2) to (3) of Theorem \ref{eq-compact-scattered},
we easily can obtain that $K$ satisfies (3) of this theorem, so $K$ is
scattered and hence by Lemma \ref{minimal-one-dimensional}, the algebra $C(K)$ and consequently $\mathcal C$
has a minimal projection.

Thus we proved that all unital subalgebras of algebras satisfying (\ref{ec-sequence-old}) have
minimal projections. Now suppose that 
$\mathcal C$ is a nonunital subalgebra of $\A$. 
If $\A$ is unital, extend $\mathcal C$ to the unital subalgebra of
$\A$ by adding the unit of $\A$ and apply Lemma \ref{atoms-unitization}. If
$\A$ is not unital, first consider  its minimal unitization $\widetilde\A$ of $\A$ and note that 
it satisfies (\ref{ec-sequence-old}) with the sequence $(\J_\alpha)_{\alpha \leq  m(\A)}$ extended by
one biggest term $\J_{m(\A)+1}=\widetilde\A$ and now follow the previous case.

(\ref{ec-subalgebra}) $\Rightarrow$
(\ref{ec-no-interval})
It is clear as $C_0((0,1])$ has no (minimal) projections.

(\ref{ec-no-interval}) $\Rightarrow$
(\ref{ec-spectrum})
Assume that (\ref{ec-spectrum}) does not hold 
and take a self-adjoint element $h$ of $\mathcal A$ with uncountable spectrum $X\subseteq \C$.
The algebra generated by $h$
is isomorphic  to
$C_0(X\cup\{0\})=\{f\in C_0(X\cup\{0\}): \lim_{x\rightarrow 0}f(x)=0\}$ 
where $X\cup\{0\}$ is compact (Corollary 1.2.3 of \cite{sakai}).
As $X$ is uncountable, there is $r>0$ such that $X\setminus (-r, r)$ is uncountable and
hence nonscattered and so there is
 a continuous surjection $\phi: (X\cup\{0\})\rightarrow [0,1]$ such that $\phi(0)=0$.
Therefore $C_0((0,1])$ embeds into $C_0(X\cup\{0\})$ and consequently into $\mathcal A$.

(\ref{ec-spectrum}) $\Rightarrow$ (\ref{ec-homomorphic})
Let $\pi: \A\rightarrow \mathcal B$ be a surjective $*$-homomorphism.
If $\A$ is unital, so is $\B$ (and both are equal to their minimal unitizations, respectively). If $\A$ is not unital, then there is 
a surjective $*$-homomorphism $\pi_1:\widetilde \A\rightarrow\widetilde \B$
sending $\lambda1_{\widetilde\A}+a$ to $\lambda1_{\widetilde\B}+\pi(a)$ for all $\lambda\in \C$
and $a\in \A$, so in both cases we have a surjective $\pi_1:\widetilde \A\rightarrow\widetilde \B$.

Note also that  $(\lambda1_{\widetilde\A}+a)-\lambda'1_{\widetilde\A}$ is invertible
 if and only if  $a-(\lambda'-\lambda)I_{\widetilde\A}$ is,
i.e., the spectrum of $\lambda1_{\widetilde\A}+a$ is a translation of the spectrum of $a$
and hence the spectra of all elements of $\widetilde\A$ are countable. 

If $a-\lambda1_{\widetilde\A}$ is invertible, then $\pi_1(a)-\lambda1_{\widetilde\B}$ is invertible as well,
so the spectrum of $\pi_1(a)$ in ${\widetilde\B}$ is not bigger than the spectrum of
$a$ in ${\widetilde\A}$, hence the spectra of all elements of  in ${\widetilde\B}$ are countable.

Applying Lemmas \ref{none-interval} and \ref{many-interval} in the unitization
of $\widetilde \B$ where there is at least one projection equal to the unit, we obtain 
either a minimal projection or a copy of $C([0,1])$. The latter is impossible 
because many elements of $C([0,1])$ have uncountable spectra. So $\B$ has a minimal projection,
a required.
\end{proof}

Before the next lemma recall Theorem \ref{eq-compact-scattered}.

\begin{lemma}\label{c(k)-ideals} Suppose that $K$ is a scattered compact Hausdorff space
with Cantor-Bendixson sequence $(K^{(\alpha)})_{\alpha\leq ht(K)}$. Then
$C(K)$ is a commutative scattered $C^*$-algebra with the Cantor-Bendixson
sequence $(\I_\alpha)_{\alpha\leq ht(C(K))}$ satisfying $ht(C(K))=ht(K)$ and
$$\I_\alpha=C^{(\alpha)}(K)=\{f\in C(K): f|K^{(\alpha)}=0\},$$
$$\I^{At}(C(K)/C^{(\alpha)}(K)) \cong c_0(K^{(\alpha)}\setminus K^{(\alpha+1)}).$$
\end{lemma}

\begin{definition}
We call a $C^*$-algebra $\mathcal A$ atomic if and only if $\I^{At}(\mathcal A)$ is an essential ideal.
\end{definition}

The following corresponds to the commutative fact that
atoms are dense in superatomic Boolean algebras:

\begin{proposition}\label{scattered-are-atomic} Every  
 scattered $C^*$-algebra is atomic.
\end{proposition}
\begin{proof} Let $\A$ be a scattered $C^*$-algebra.
It is easy to check that $\mathcal I^\perp$ 
is a (closed) ideal of $\mathcal A$ for any ideal $\mathcal I\subseteq \mathcal A$.
Applying Proposition \ref{intersection-of-atomic} for $\mathcal J=\I^{At}(\mathcal A)^\perp$
we conclude that $\I^{At}(\mathcal J)=\{0\}$, which  by Theorem \ref{theorem1}
(\ref{ec-subalgebra}) means  that $\mathcal J=0$.  Therefore
$\I^{At}(\mathcal A)$ is essential.
\end{proof}

\begin{proposition}\label{sub} Every subalgebra and every
$*$-homomorphic image of a scattered $C^*$-algebra is 
a scattered $C^*$-algebra.
\end{proposition}
\begin{proof}
Condition (1) of Theorem \ref{theorem1} is hereditary with respect to $*$-homomorphic
images and condition (4) is hereditary with respect to subalgebras.
\end{proof}

By locally finite dimensional (LF) we will mean a $C^*$-algebra in which
every finite set of elements can be approximated from a finite dimensional subalgebra,
for separable algebras it is equivalent to be AF  which means that there is a directed family
of finite dimensional subalgebras with dense union
(see \cite{farah-katsura} for detailed discussion of versions of the definition
of nonseparable AF algebras).  Recall also that a  $C^*$-algebra $\A$ is
called GCR (type I) if and only if $CCR(\A/\J)\not=\{0\}$ whenever $\J\subseteq \A$
is an ideal of $\A$ satisfying $\J\not=\A$.
The following, among others, corresponds to the fact (Proposition \ref{0-dim}) that
scattered compact Hausdorff spaces are zero-dimensional. Of course since all scattered $C^*$-algebras
are LF $C^*$-algebras (Lemma 5.1 of \cite{lin}) which have real rank zero 
(V. 7. 2 of \cite{davidson}), so do scattered $C^*$-algebras. However, this
can be also observed independently.

\begin{proposition}\label{real-rank-zero}\cite{lin, kusuda-af, tomiyama}  Suppose that $\A$ is a 
scattered $C^*$-algebra. Then 
\begin{enumerate}
\item $\A$ has  real rank zero,
\item $\A$ is LF,
\item $\A$ is a GCR algebra.
\end{enumerate}
\end{proposition}
\begin{proof} For (1) will use one of the equivalent conditions to being of the real rank zero, which says that 
 self-adjoint elements of the algebra can be approximated by self-adjoint elements
which have finite spectrum (V. 7. 3 of \cite{davidson}).  Suppose that $a$ is a self-adjoint element of $\A$.
Let $K$ be compact Hausdorff such that the unital $C^*$-algebra generated 
in $\mathcal A$ by
$a$ is isomorphic to $C(K)$. By Theorem \ref{theorem1} (\ref{ec-subalgebra}) 
$K$ must be scattered and so the element corresponding to $a$ in $C(K)$ must be
approximated by a  linear combination of projections which have spectrum $\{0,1\}$
and are self adjoint. 

To prove (2) we need to use, as in \cite{lin}, the fact due to Effros,
that  a separable $C^*$-algebra $\A$ is approximately finite if for an ideal $\J\subseteq \A$,  both $\J$ and the 
quotient $\A/\J$ are
approximately finite (III. 6. 3. \cite{davidson}) and
the fact that  a $C^*$-algebra is LF if all separable subalgebras  
are LF, which follows directly from our definition. By Proposition \ref{sub}
any  separable subalgebra of $\A$ is scattered, so one can apply the above fact
and the transfinite induction  using either (2) or (3) of Theorem \ref{theorem1}
and  obtain condition (2).

To prove (3) let $\J\subseteq A$ be an ideal of $\A$ satisfying $\J\not=\A$.
Then $\A/\J$ is scattered by  Propositon \ref{sub} and so $\I^{At}(\A/ \J)\not=\{0\}$.
Now apply Proposition \ref{ccr} to conclude that $CCR(\A/ \J)\not=\{0\}$.
\end{proof}

\begin{proposition}[\cite{kusuda-af}]\label{commutative-af}
The following are equivalent for a $C^*$-algebra $\A$:
\begin{enumerate}
\item $\A$ is scattered,
\item Every subalgebra of $\A$ is LF,
\item Every commutative (and separable) subalgebra of $\A$ is scattered,
\end{enumerate}
\end{proposition}
\begin{proof} 
(1) $\rightarrow$ (2), (3). This follows from Propositions \ref{sub} and \ref{real-rank-zero} (2).
The fact that $(2)$ or $(3)$ imply $(1)$ follows from
Theorem \ref{theorem1} (\ref{ec-no-interval}), as $C_0((0,1])$ is neither scattered nor $LF$.
\end{proof}

\begin{proposition}\label{not-gcr}
There are scattered $C^*$-algebras $\A$ such that the ideal $CCR(\A)$ does
not appear in  the Cantor-Bendixson sequence $(\I_\alpha)_{\alpha\leq ht(\A)}$.
\end{proposition}
\begin{proof} Let $\A$ be the algebra $C(\omega+1, \widetilde{\K(\ell_2))}\cong
 C(\omega+1)\otimes\widetilde{\K(\ell_2)})$
of all continuous complex-valued functions on the one point compactification of the countable infinite
discrete space into the unitization of the algebra of compact operators on 
a separable Hilbert space. Let $(I_\alpha)_{\alpha\leq ht(\A)}$ be its Cantor-Bendixson sequence of $\A$.
Note that $\{f\in \A: f(n)\in \K(\ell_2)),\ f(\omega)=0\}\cong c_0\otimes\K(\ell_2)$ 
is an essential ideal isomorphic to the countable direct sum of the algebras 
of compact operators and so it is $\I_1=\I^{At}(\A)$ by Proposition \ref{essential-compact}.

Moreover
$[p_n]_{\I^{At}(\A)}$ for $p_n=\chi_{\{n\}}\otimes Id_{\B(\ell_2)}$ is a minimal projection in $A/{\I^{At}(\A)}$
for each $n\in \omega$, so $p_n\in \I_2\setminus\I_1$. 
However $p_n\not\in CCR(\A)$ as the irreducible representation $\pi_n:\A\rightarrow \B(\ell_2)$
given by $\pi(a)=a(n)$ sends $p_n$ to a noncompact operator. 

As for any one-dimensional projection $q\in \K(\ell_2)$ the element $[a_q]_{\I^{At}(\A)}$  is 
a minimal projection in  $\A/{\I^{At}(\A)}$ as well, where $a_q=\chi_{\omega+1}\otimes q$,
it follows that $\A/\I_2$ is isomorphic to $\C$ and so $\I_3=\A$ and $ht(\A)=3$.

Now consider any irreducible representation $\pi: \A\rightarrow \B(\HH)$ for some Hilbert space $\HH$
aiming at proving that $a_q$ is in $CCR(\A)$ for each  one-dimensional projection $q\in \K(\ell_2)$.
Consider the canonical copy of $C(\omega+1)$ in $\A$ which is the center of $\A$
so ends up in the center of $\pi[\A]$. But the irreducibility of
$\pi$ gives that the commutator of $\pi[\A]$ is in $\C Id_{\HH}$ (5.1.5 \cite{murphy}),
so $C(\omega+1)$ is sent into $\C Id_{\HH}$. It follows that
 $Ker(\pi)\cap C(\omega+1)$ is 
 one of the ideals $\{f\in C(\omega+1): f(\xi)=0\}$ for some $\xi\in \omega+1$,  but
it contains an approximate unit of 
the ideal $\{f\in \A: f(\xi)=0\}$, and so  $Ker(\pi)$
 is included  in  $\{f\in \A: f(\xi)=0\}$ for some $\xi\in \omega+1$. It follows that $\pi$ corresponds to one of the 
irreducible representations of $\A/\{f\in \A: f(\xi)=0\}\equiv \widetilde{\K(\ell_2)}$, but
all such representations send $a_q$ onto a compact operator so $a_q$ is in $CCR(\A)$ as required.

It follows that $CCR(\A)$ cannot  be any of the ideals  $\I_1$, $\I_2$
nor  $\I_3$.

\end{proof}

One should note that being a scattered $C^*$-algebra is  equivalent to many
other conditions which as in the commutative version,  that is Theorem \ref{ec-banach},
are concerned with the dual and the bidual of the algebra:

\begin{theorem}[\cite{jensen1}, \cite{jensen2}]\label{ecStar-banach}
Suppose that $\mathcal A$ is a $C^*$-algebra and $\A''$ is its enveloping
von Neumann algebra.  The following
conditions are equivalent  for $\A$ to be scattered:
\begin{enumerate}
\item\label{ec-enveloping-algebra} $\mathcal A''= \prod_{\alpha<\kappa} \mathcal B(\mathcal H_\alpha)$
for some cardinal $\kappa$.
\item\label{ec-projectionsvN} Each non-zero projection in $\mathcal A''$ majorizes a minimal projection in $\mathcal A$.
\item \label{ec-non-degenerate-representations} Every non-degenerate representation of $\mathcal A$ is a
sum of irreducible representations.
\item\label{ec-atomic-states} Every positive functional  $\mu$ on $\mathcal A$
is of the form $\Sigma_{n\in \N}t_n\mu_n$ where $\mu_n$s are
pure states and $t_n \in \R^+ \cup \{0\}$ are such that $\Sigma_{n\in \N} t_n <\infty$.
\item \label{ec-trace-class}  The dual of $\mathcal A$ is isomorphic to the $\ell_1$-sum
 $(\bigoplus_{\alpha<\beta} \mathcal TC(\mathcal H_\alpha))_{\ell_1}$
where $TC(\mathcal H)$ denotes the trace class operators on a Hilbert space $\mathcal H$.
\item \label{subalgebra-conjugate} The dual spaces $\mathcal C^*$ of separable subalgebras
 $\mathcal C\subseteq \mathcal A$ are separable.
\item \label{subalgebra-spectrum} The spectrum spaces $\hat {\mathcal C}$
of separable subalgebras $\mathcal C\subseteq \mathcal A$
are countable.
\end{enumerate}
\end{theorem}

\section{Cantor-Bendixson derivatives of tensor products}

Despite the following theorem of C. Chu it is unclear how to calculate
the Cantor-Bendixson sequence of the tensor product of scattered $C^*$-algebras. 
Note that even in the commutative case $(K\times L)'$ is not $K'\times L'$,
for $K, L$ compact,
however, $K\times L\setminus (K\times L)'=(K\setminus K')\times (L\setminus L')$.

\begin{theorem}[\cite{chu-crossed}] \label{tensor-scattered} 
Let $\A$ and $\B$ be $C^*$-algebras.
The maximal tensor product $\A \otimes_{\max} \B$ is scattered if and
only if $\A$ and $\B$ are scattered.
\end{theorem}
By Proposition \ref{real-rank-zero} (3)  every scattered $C^*$-algebra 
is GCR (or type $I$ or postliminar; see \cite[Theorem 1.5.5]{invitation} and page 169 of \cite{murphy}). Therefore every 
scattered $C^*$-algebra $\A$ is nuclear (e.g., \cite[IV.3.1.3]{Blackadar}), i.e., for every $C^*$-algebra $\B$ 
there is a unique $C^*$-tensor product $\A \otimes \B$.

\begin{lemma}\label{atoms-tensor-product}
Assume $\mathcal A_1$ and $\mathcal A_2$ are two scattered $C^*$-algebras. Then 
$$ \I^{At}(\mathcal A_1 \otimes \mathcal A_2)= \I^{At}(\mathcal A_1) \otimes \I^{At}(\mathcal A_2).$$
\end{lemma}
\begin{proof}
Let $p_i$ be a minimal projection in $\mathcal A_i$ for $i=1,2$. Then 
$$p_1 \otimes p_2 (\mathcal A_1 \otimes \mathcal A_2) p_1 \otimes p_2 = 
p_1 \mathcal A_1  p_1 \otimes p_2 \mathcal A_2 p_2= \mathbb C (p_1  \otimes p_2). $$
Therefore $p_1 \otimes p_2$ is a minimal projection in $\mathcal A_1 \otimes \mathcal A_2$ and 
hence $ \I^{At}(\mathcal A_1) \otimes \I^{At}(\mathcal A_2) \subseteq \I^{At}(\mathcal A_1 \otimes \mathcal A_2)$.

Since every scattered $C^*$-algebra is atomic (Proposition \ref{scattered-are-atomic}) we know that
$\I^{At}(\A_1)$ and $\I^{At}(\A_2)$ are essential ideals of $\A_1$ and $\A_2$, respectively.
If we show that $\I^{At}(\A_1) \otimes \I^{At}(\A_2)$ is an essential ideal of $\A_1 \otimes \A_2$
then by the above and Proposition \ref{essential-maximal} we have 
$\I^{At}(\mathcal A_1) \otimes \I^{At}(\mathcal A_2) = \I^{At}(\mathcal A_1 \otimes \mathcal A_2)$.
Assume $\J$ is a nonzero  ideal of $\A_1 \otimes \A_2$. It is known that $\J$ contains a nonzero
elementary tensor product $x_1 \otimes x_2$, where $x_1 \in \A_1$ and $x_2 \in \A_2$
(this is true in general for nonzero
 closed ideals of
 the minimal tensor products of $C^*$-algebras; see for example \cite[Lemma 2.12]{Blanchard-Kirchberg}).
 For $i=1,2$,
since $\I^{At}(\A_i)$ is an essential ideal of $\A_i$, there is $a_i\in \I^{At}(\A_i)$ such that 
$a_i x_i \neq 0$.   
 Then 
$$
 a_1 x_1 \otimes a_2 x_2= (a_1 \otimes a_2) (x_1 \otimes x_2) \neq 0
$$
belongs to $\J\cap \I^{At}(\A_1) \otimes \I^{At}(\A_2)$. Thus $\I^{At}(\A_1) \otimes \I^{At}(\A_2)$ is an essential ideal
of $\A_1 \otimes \A_2$.

\end{proof}

\begin{proposition}\label{cb-tensor} Suppose that $\A$ is a scattered $C^*$-algebra 
with the Cantor-Bendixson sequence $(\I_\alpha)_{\alpha\leq ht(\A)}$.
Then $\A\otimes \K(\ell_2)$ is a scattered $C^*$-algebra 
whose Cantor-Bendixson sequence $(\J_\alpha)_{\alpha\leq ht(\A)}$ satisfies
$\J_\alpha=\I_\alpha\otimes \K(\ell_2)$ for every $\alpha\leq ht(\A\otimes \K(\ell_2))=ht(\A)$.
In particular $\J_{\alpha+1}/\J_\alpha$ is $*$-isomorphic to $(\I_{\alpha+1}/\I_\alpha)\otimes \K(\ell_2)$.
\end{proposition}
\begin{proof}
We prove $\J_\alpha=\I_\alpha\otimes \K(\ell_2)$ by induction on $\alpha\leq ht(\A)$.
For $\alpha=1$ this follows from Lemma \ref{atoms-tensor-product}.
At a successor ordinal by
Lemma \ref{atoms-tensor-product} and the 
inductive assumption we have $\J_{\alpha+1}/\J_{\alpha}=\I^{At}\big((\A\otimes \K(\ell_2)\big)/\J_\alpha)
=\I^{At}(\big(\A\otimes\K(\ell_2)\big)/(\I_\alpha\otimes \K(\ell_2)))
=\I^{At}\big((\A/\I_{\alpha})\otimes\K(\ell_2)\big)=\I^{At}(\A/\I_{\alpha})\otimes\K(\ell_2)
=(\I_{\alpha+1}/\I_{\alpha})\otimes\K(\ell_2)$ and so $\J_{\alpha+1}=\I_{\alpha+1}\otimes \K(\ell_2)$.
The limit ordinal case is immediate.
\end{proof}

\section{Fully noncommutative scattered $C^*$-algebras}

\begin{definition}\label{fully} Let $\A$ be a  scattered $C^*$-algebra 
with the Cantor-Bendixson sequence $(\I_\alpha)_{\alpha\leq ht(\A)}$.
$\A$
is called fully noncommutative
if  $\I_{\alpha+1}/\I_\alpha$ is isomorphic to the algebra of all compact operators 
on a Hilbert space, for each $\alpha<ht(\A)$.
\end{definition}

\begin{lemma}\label{all-ideals} Suppose that $\mathcal A$ is a 
fully noncommutative scattered $C^*$-algebra 
with the Cantor-Bendixson sequence $(\I_\alpha)_{\alpha \leq ht(\A)}$ and $\J\subseteq \A$ 
is an ideal of $\A$. Then  $\J=\I_\alpha$ for some $\alpha\leq ht(\A)$.
\end{lemma}
\begin{proof} Let $\beta\leq ht(\A)$ be the  minimal ordinal such that
$\I_\beta\not\subseteq \J$. If there
is no such $\beta$ we have $\I_{ht(\A)}=\A=\J$.
It follows that $\beta=\alpha+1$ for some $\alpha< ht(\A)$.
So $\sigma_\alpha[\I_{\alpha+1}\cap \J]$ is a proper ideal of $\I_{\alpha+1}/\I_\alpha$,
where $\sigma_\alpha$ is the quotient map.
By Definition \ref{fully} the quotient
$\I_{\alpha+1}/\I_\alpha$ is isomorphic to the algebra of all compact
operators on a Hilbert space, which has no nonzero proper ideals.
So $(\J\cap \I_{\alpha+1})/\I_\alpha$ is the zero ideal.
Moreover by Theorem \ref{theorem1} (\ref{ec-sequence-new}) 
$\I^{At}(\A/\I_\alpha)=\I_{\alpha+1}/\I_\alpha$, which is an essential ideal by
Theorem \ref{theorem1} (\ref{ec-homomorphic}) and 
Proposition \ref{scattered-are-atomic}. Therefore $\J/\I_\alpha=\{0\}$
which implies that $\J=\I_\alpha$, as required.
\end{proof}

\begin{proposition}\label{chain} Suppose that $\A$ is
a scattered $C^*$-algebra. 
The following are equivalent:
\begin{enumerate}
\item $\A$ is fully noncommutative,
\item the ideals of $\A$ form a chain,
\item the centers of the  multipler algebras of
any quotient of $\A$ are all trivial.
\end{enumerate}
\end{proposition}
\begin{proof}
The implication (1) $\Rightarrow$ (2)  is Lemma \ref{all-ideals}. 

(2) $\Rightarrow$ (1) Let $(\I_\alpha)_{\alpha\leq ht(\A)}$ be 
the Cantor-Bendixson sequence of $\mathcal A$.
Suppose that one of the quotients $\I_{\alpha+1}/\I_\alpha$ is not  isomorphic to 
the algebra of all compact operators on a Hilbert space. 
Since $\I_{\alpha+1}/\I_\alpha= \I^{At}(\A/ \I_\alpha)$ and by 
Proposition \ref{atoms-are-compact},  $\I_{\alpha+1}/\I_\alpha$
 it is isomorphic to
some subalgebra of all compact operators on a Hilbert space (\cite[Theorem 1.4.5.]{invitation}).
Hence, we may assume that $\I_{\alpha+1}/\I_\alpha$ is isomorphic to
a subalgebra of an algebra of the form $\K(\HH_1)\oplus\K(\HH_2)$ for
some Hilbert spaces $\HH_1$ and $\HH_2$.
Now $\I_\alpha+\K(\HH_1)\oplus\{0\}$ and $\I_\alpha+\{0\}\oplus\K(\HH_2)$
are two incomparable (with respect to inclusion) ideals of $\A$.

(1) $\Rightarrow$ (3)  By Lemma \ref{all-ideals} a nonzero quotient
of $\A$ must be of the form $\A/\I_\alpha$ for some $\alpha<ht(\A)$.
By the Dauns-Hofmann theorem (4.4.8 of \cite{pedersen}) the center of the multiplier algebra of
$\A/\I_\alpha$ is isomorphic to the $C^*$-algebra of the
bounded continuous functions on the spectrum $\hat\A$ of $\A$. The quotient
$\A/\I_\alpha$ is again scattered  and therefore $\I_{\alpha+1}/\I_\alpha$ is an
essential ideal of $\A/\I_\alpha$ (Proposition \ref{scattered-are-atomic}). This implies that
$\A/\I_\alpha$ can be embedded into the multiplier algebra
of $\I_{\alpha+1}/\I_\alpha\cong \K(\ell_2(\kappa))$, for some cardinal $\kappa$
so that the embedding is the identity on $\I_{\alpha+1}/\I_\alpha$.
It follows that $\A/\I_\alpha$ has a faithful representation in
$\B(\ell_2(\kappa))$, where $\I_{\alpha+1}/\I_\alpha$
is mapped onto $\K(\ell_2(\kappa))$. 
So this representation is irreducible and its kernel zero is included in the kernel
of any other irreducible representation of  $\A/\I_\alpha$. Therefore all points
of the spectrum of $\A/\I_\alpha$ are in the closure (with the Jacobson topology)
of this one point. Hence the only continuous
functions on the spectrum of $\A/\I_\alpha$ are constant maps.
It follows from the Dauns-Hofmann theorem that the center
of the multipliers of  $\A/\I_\alpha$ is trivial.

(3) $\Rightarrow$ (1) Again let $(\I_\alpha)_{\alpha\leq ht(\A)}$ be 
the Cantor-Bendixson sequence of $\mathcal A$. If $\A$ is not fully commutative, then
it has a quotient $\A/\I_\alpha$ with an essential
ideal $\I_{\alpha+1}/\I_\alpha$ which is isomorphic to
$\bigoplus_{i\in I} \K(\HH_i)$ for some nonzero Hilbert spaces $\HH_i$ and some
 set $I$ with at least two elements. By the essentiality of the ideal $\I_{\alpha+1}/\I_\alpha$
  in $\A/\I_\alpha$,
there is an embedding of $\A/\I_\alpha$ into
$\bigoplus_{i\in I}\B(\HH_i)$, which sends $\I_{\alpha+1}/\I_\alpha$
onto $\bigoplus_{i\in I} \K(\HH_i)$ and $\bigoplus_{i\in I}\B(\HH_i)$ can be identified
with the algebra of multipliers of $\I_{\alpha+1}/\I_\alpha$.
However since $|I|\geq 2$, the projections on the factors $\B(\HH_i)$ witness the fact that
the center of the algebra of multipliers of $\A/\I_\alpha$ is nontrivial.
\end{proof}

Given a GCR $C^*$-algebra $\A$ by its GCR composition series we mean a 
sequence of ideals $(\J_\alpha)_{\alpha\leq \alpha_0}$ for some ordinal $\alpha_0$
such that $\J_0=\{0\}$, $\J_{\alpha_0}=\A$, $\J_\lambda=\overline{\bigcup_{\alpha<\lambda}\J_\alpha}$
for each limit ordinal $\lambda\leq \alpha_0$ and 
$$\{[a]_{\J_\alpha}: a\in J_{\alpha+1}\}=CCR(\J_{\alpha+1}/J_\alpha),$$
for each $\alpha<\alpha_0$ (see Section 1.5 of \cite{invitation}).

\begin{proposition}\label{gcr-fully}
Suppose that $\A$ is a fully noncommutative scattered $C^*$-algebra.
Then its Cantor-Bendixson sequence and its GCR composition series coincide.
\end{proposition}
\begin{proof} Let $(I_\alpha)_{\alpha\leq ht(\A)}$ be the Cantor-Bendixson sequence
of $\A$.
By induction on $\alpha\leq ht(\A)$ we prove that
$\I_\alpha=\J_\alpha$. The limit ordinal case is trivial.  So let $\alpha<ht(\A)$
and consider $\A/\I_\alpha$. Since $\A$ is fully noncommutative
$\I^{At}(\A/\I_\alpha)$ is isomorphic to the algebra $\K(\ell_2(\kappa))$
for some cardinal $\kappa$. Since $\I^{At}(\A/\I_\alpha)$  is essential
in $\A/\I_\alpha$  by Proposition \ref{scattered-are-atomic}, we conclude
by Proposition \ref{embedding-scattered}  that there is an embedding
$i: \A/\I_\alpha\rightarrow \B(\ell_2(\kappa))$ such that
$i[\I^{At}(\A/\I_\alpha)]=\K(\ell_2(\kappa))$. Since
$\K(\ell_2(\kappa))$ is irreducible in $\B(\ell_2(\kappa))$ it follows that
$i$ is an irreducible representation of $\A/\I_\alpha$.
So $CCR(\A/\I_\alpha)\subseteq \I^{At}(\A/\I_\alpha)$.
The other inclusion holds in general by Proposition \ref{ccr}.
\end{proof}

\section{Some constructions of scattered  $C^*$-algebras}

In this section we give examples of constructions
of scattered $C^*$-algebras with some prescribed sequences of Cantor-Bendixson derivatives.
All these examples correspond to classical classes of scattered locally compact spaces.
As we want to produce ``genuinely" noncommutative examples, we will focus on
two notions of being fully noncommutative (Definition \ref{fully}) and the stability of
$C^*$-algebras (cf. \cite{rordam-stable}):

\begin{definition} A $C^*$-algebra $\A$ is called stable if and only if
$\A$ is $*$-isomorphic to $\A\otimes \K(\ell_2)$.
\end{definition}

The first immediate group of constructions is obtained by tensoring
the commutative examples by $\K(\ell_2)$ and applying Proposition \ref{cb-tensor}.
 They are not fully noncommutative,
but are stable.
 Before the following theorem recall
the definitions of width and height of a Scattered $C^*$-algebras (Definition \ref{width-height}).

\begin{theorem}\label{consistency} Let $\kappa$ be a regular cardinal.
\begin{itemize}
\item There are stable scattered $C^*$-algebras
of countable width and height $\alpha$ for any $\alpha<\omega_2$.
\item It is consistent that there are stable scattered $C^*$-algebras
of countable width and height $\alpha$ for any $\alpha<\omega_3$.
\item It is consistent that there are stable scattered $C^*$-algebras
of  width $\kappa$ and height $\kappa^+$.
\item There are stable scattered $C^*$-algebras $\A$ of height 2
with $\I^{At}(\A)\cong c_0\otimes\K(\ell_2)$ and $\A/\I^{At}(\A)\cong c_0(\cc)\otimes\K(\ell_2)$
where $\cc$ denotes the cardinality of the continuum.
\end{itemize}
\end{theorem}
\begin{proof} In all these cases we apply Proposition \ref{cb-tensor}
for tensor products $C(K)\otimes \K(\ell_2)$ for appropriate
compact scattered Hausdorff spaces $K$ from \cite{juhasz-weiss},
\cite{bs}, \cite{martinez-jsl}, \cite{koepke} and the $\Psi$-spaces
surveyed e.g. in \cite{hrusak}.
See also \cite{roitmanhandbook}.
\end{proof}

In the rest of this section we focus on obtaining fully noncommutative examples
which do not require any special set-theoretic assumptions. Before doing so
let us inquire about the relationship between the stability
and being fully noncommutative. As shown in the above examples the former does not
imply the latter, but being fully noncommutative implies the stability
for separable  scattered $C^*$-algebras which have no unital quotients. 
The condition of  having no unital quotients is necessary as
quotients of stable algebras are stable (2.3 (ii) of \cite{rordam-stable})
and no stable algebra can be unital.
%To see this we need the following elementary observation.

%\begin{lemma}\label{stable-elementary}
%A  subalgebra $\A$ of compact operators of the form
%$\bigoplus_{i\in I}\K(\HH_i)$ is stable if and only if the Hilbert spaces $\HH_i$
%are infinite dimensional for all $i\in I$.
%\end{lemma}
%\begin{proof}
%For the forward implication, note that if $\A$ is stable, then so is each of 
%its quotients (2.3. (ii) of \cite{rordam-stable}),
%but obviously $\K(\HH_i)$ is not stable if $\HH_i$ is finite dimensional. 
%For the other direction, $(\bigoplus_{i\in I}\K(\HH_i))\otimes \K(\ell_2)$ is isomorphic to
 %the algebra $\bigoplus_{i\in I}\K(\HH_i\otimes \ell_2)$, which should be isomorphic to
%$\bigoplus_{i\in I}\K(\HH_i)$, and therefore each $\HH_i$ is infinite dimensional.
%\end{proof}

\begin{lemma}\label{fnoncom-stable} Suppose that $\A$ is a 
 separable   scattered $C^*$-algebra with the Cantor-Bendixson sequence
$(\I_\alpha)_{\alpha<ht(\A)}$. Then $\A$ is stable if and
only if  $\I_{\alpha+1}/\I_\alpha$ is stable for each $\alpha<ht(\A)$.
In particular, if 
 $\A$ is fully noncommutative and with no unital quotient, then $\A$  is stable.

\end{lemma}
\begin{proof} 
For the forward implication uses the fact that ideals and quotients of
stable $C^*$-algebras are stable (2.3. (ii) of \cite{rordam-stable}).

The proof of the backward implication is by induction on the height of the algebra.
Note that this height must be a countable ordinal as the algebra is assumed to be separable.
Suppose that the height is a successor ordinal $\alpha+1$.
We have that $\A=\I_{\alpha+1}$ and that  $\I_\alpha$ is stable by the inductive
assumption. The quotient
 $\I_{\alpha+1}/\I_\alpha$  is stable by the hypothesis.

Consider the short exact sequence,
$$0\rightarrow \I_\alpha\rightarrow \A\rightarrow \I_{\alpha+1}/\I_\alpha\rightarrow0.$$
Now we use the fact that scattered $C^*$-algebras are approximately
finite (Lemma 5.1 of \cite{lin}) and Blackadar's characterization of separable stable AF-algebras
as AF-algebras with no nontrivial bounded trace (\cite{blackadar-af},
\cite{rordam-stable}). A nontrivial bounded trace on $\A$ would need
to be zero on $\I_\alpha$, by the inductive assumption. Hence it would define a 
nontrival bounded trace on $\I_{\alpha+1}/\I_\alpha$, which is impossible
since  it is  stable
(see also Proposition 6.12 of \cite{rordam-stable}).
If the height of $\A$ is a limit ordinal, then the result follows from the fact that countable inductive limits 
of separable stable $C^*$-algebras are also stable (Corollary 2.3 of \cite{rordam-stable}). 

To conclude the last part of the lemma from the previous, note that the assumption
that $\A$ is fully noncommutative and that the quotients $\I_{\alpha+1}/\I_\alpha$
are nonunital, implies that they are isomorphic to the algebras of all compact operators
on an infinite dimensional Hilbert spaces, and such algebras are stable.
\end{proof}

Examples from papers \cite{psi-space, thin-tall} show that nonseparable 
fully noncommutative scattered $C^*$-algebras do not have to be stable.
However by tensoring a fully-noncommutative $C^*$-algebra by $\K(\ell_2)$
we obtain a fully noncommutative $C^*$-algebra of the same height and width
which is additionally stable. 

Recall that for a cardinal $\kappa$, $\kappa^+$ denotes the minimal cardinal which is strictly bigger than the
$\kappa$.
The following proposition shows that one can start from $\K(\ell_2(\kappa))$,
where $\kappa$ is a regular cardinal and increase the height of the algebra
up to any ordinal $\theta<\kappa^+$. In these constructions the algebras
of lower height are not essential ideals of those previously constructed, and so we can not
get to the height $\kappa^+$ without increasing the width. These constructions correspond
to consecutive application of one-point compactification and the Cartesian product 
by $\kappa$ with the discrete topology, that is taking the disjoint union of $\kappa$-many
copies of one point compactification of the previous spaces.

\begin{proposition} \label{kappa-going-up}
Suppose that $\kappa$
is a regular cardinal. For any $\theta<\kappa^+$ there is a scattered $C^*$-algebra 
$\A_\theta$ of height $\theta$  such that if  $(\mathcal I_{\alpha}^\theta)_{\alpha\leq ht(\mathcal A_{\theta})}$ is the
Cantor-Bendixson 
sequence of $\mathcal A_\theta$,  the algebra $\mathcal I_{\alpha+1}^\theta/\mathcal I_\alpha^\theta$ 
is  isomorphic
to $\mathcal K(\ell_2 (\kappa))$ for every $\alpha<ht(\A_\theta)$.
In particular, each $\A_\theta$ is fully noncommutative,
with $wd(\A_\theta)=\kappa$. There are stable examples with this property.
\end{proposition}
\begin{proof}
Let $\{1_{\beta,\alpha} : \alpha, \beta < \kappa\}$ be a system
 of matrix units for $\mathcal K(\ell_2 (\kappa))$.
Define $\mathcal A_\theta$  for $\theta<\kappa^+$ by recursion
together with embeddings $\pi_{\theta,\eta}: \mathcal A_\eta \rightarrow \mathcal A_\theta$
for $\eta \leq\theta<\kappa^+$. The embeddings  satisfy the condition that
 $\pi_{\gamma,\theta}\circ \pi_{\theta,\eta}=\pi_{\gamma, \eta}$ 
for $\eta\leq\theta\leq\gamma<\kappa^+$.

Let $\mathcal A_0=\{0\}$ and $\mathcal A_1=\mathcal K( \ell_2 (\kappa))$.
Fix $\theta < \kappa^+$ and suppose we are given $\mathcal A_\eta$ and $\pi_{\gamma,\eta}$
 for every  $\eta \leq\gamma<\theta$.
First consider the case where $\theta=\eta+1$ for some $\eta<\kappa^+$.
Let $\mathcal A_\theta = \widetilde{A_\eta} \otimes \mathcal K (\ell_2 (\kappa))$.
%Another way to describe $\mathcal A_{\theta}$ is that it is the $C^*$-algebra generated by the following two sets
 %(assume that $\mathcal A_\eta\subseteq \mathcal B(\mathcal H_\eta)$
%for some Hilbert spaces $H_\eta$ for all $\eta<\theta<\kappa^+$).
%\begin{enumerate}
%\item  The 
%algebra  $\bigoplus_{\alpha<\kappa} \mathcal A_{\eta, \alpha}^\circledast$, where if
%$\mathcal H_{\eta,\alpha}=\mathcal H_\eta$ then
%$\mathcal A_{\eta, \alpha}=\mathcal A_\eta $ is considered as a subalgebra of 
%$\mathcal B(\mathcal H_{\eta, \alpha})$ for all $\alpha< \kappa$,
%\item $\{ I_{\beta,\alpha}: \alpha, \beta < \kappa \}$, where each $I_{\beta, \alpha}$ is an operator on 
%the Hilbert space $\bigoplus_{\gamma<\kappa} \mathcal H_{\eta, \gamma}$ which sends  
%$\mathcal H_{\eta, \alpha}$ identically to $\mathcal H_{\eta, \beta}$, and sends everything else to zero.
%\end{enumerate}
%Let $\mathcal H_\theta  = \bigoplus_{\gamma<\kappa} \mathcal H_{\eta, \gamma}$ and 
%therefore $\mathcal A_{\theta} \subseteq   \mathcal B(\mathcal H_\theta)$.
 By the assumption we have $ ht(\mathcal A_{\eta})=\eta$ (i.e., $\mathcal I_{\eta}^{\eta} =\mathcal A_\eta$).
  Define the embedding  $\pi_{\theta, \eta}:\mathcal A_\eta \rightarrow 
 \mathcal A_\theta$ by mapping $a\rightarrow (a,0)\otimes 1_{0,0}$  (the top left entry of the $\kappa\times \kappa$ matrix over $\widetilde{A_\eta}$) and let $\pi_{\theta, \alpha}= \pi_{\theta, \eta}\circ \pi_{\eta,\alpha}$ for $\alpha<\eta$.
 % The image of $\pi_{\theta, \eta}$ is a closed ideal of $\mathcal A_\theta$
% which is *-isomorphic to $\mathcal A_\eta$.
  %Under this isomorphism, from Lemma \ref{intersection-of-atomic}, it is clear that $\mathcal I^{\theta}_\alpha \cong \mathcal I^{\eta}_\alpha$
%for every $\alpha \leq\eta$. 
%Also $\mathcal A_\theta / \mathcal I^{\theta}_{\eta} \cong \mathcal A_\theta / \pi_{\theta,\eta}[\mathcal A_\eta]$
 %is isomorphic to 
 %$ \mathcal K(\ell_2 (\kappa))$; the map from $\mathcal A_\theta$ into $\mathcal K (\ell_2 (\kappa))$ 
 %which sends $(a, \lambda) \otimes T$ (where $(a,\lambda)\in \mathcal A_\eta ^\circledast$ and $T\in \mathcal K (\ell_2 (\kappa))$) to $\lambda T$ linearly extends to a well-defined surjective *-homomorphism and its kernel is 
 %$\pi_{\theta,\eta}[\mathcal A_\eta]$. Therefore by Proposition \ref{atoms-are-max-ideal} 
 %$ \I^{At}(\mathcal A_\theta / \mathcal I^{\theta}_{\eta})\cong \mathcal K(\ell_2 (\kappa))$, which also implies that 
 %$\mathcal A_\theta$ is scattered and that
 %$ht(\mathcal A_\theta )= \beta +1 = \theta$. 
 
 We claim that  $\mathcal I^{\theta}_{\beta} 
   = \mathcal I^\eta_{\beta} \otimes \mathcal K(\ell_2 (\kappa))$
  for all $\beta< ht(\mathcal A_\eta) =\eta$. We show this by induction on $\beta$. Assume this is true for 
some $\beta< \eta$. Then
\begin{align}
 \nonumber \mathcal I^{\theta}_{\beta+1}/ \mathcal I^\theta_\beta =
  \I^{At}(\mathcal A_\theta/ \mathcal I^\theta_\beta)&= \I^{At}\big(\frac{\widetilde{A_\eta}\otimes 
  \mathcal K(\ell_2 (\kappa))}{\mathcal I_\beta^\eta \otimes \mathcal K(\ell_2(\kappa))}\big) \\
  \nonumber &\cong \I^{At}\big(\frac{\widetilde{A_\eta}}{\mathcal I_\beta^\eta} \otimes \mathcal K(\ell_2(\kappa))\big) \\
\tag{Lemma \ref{atoms-tensor-product}}   &= \I^{At}\big(\frac{\widetilde{A_\eta}}{\mathcal I_\beta^\eta}\big)  \otimes \I^{At} \big(\mathcal K(\ell_2(\kappa))\big)  \\
\tag{Lemma \ref{atoms-unitization}}  &=\frac{ \mathcal I^{\eta}_{\beta+1}}{\mathcal I^\eta_\beta} \otimes  \mathcal K(\ell_2(\kappa))    
    \\
    \nonumber &\cong \mathcal K(\ell_2(\kappa)\otimes \ell_2(\kappa)) \cong  \mathcal K(\ell_2(\kappa)).
\end{align}

Let $\sigma_{\theta, \beta}$ (resp. $\sigma_{\eta,\beta}$) denote the canonical quotient maps from 
$\mathcal A_\theta$ (resp. $\mathcal A_\eta$) onto $\mathcal A_\theta / \mathcal I_\beta^\theta$  
 (resp. $\mathcal A_\eta / \mathcal I_\beta^\eta$). Also let 
 $$ \psi : \frac{\mathcal A_\theta }{\mathcal I_\beta^\theta} =\frac{\widetilde{A_\eta} \otimes 
  \mathcal K(\ell_2 (\kappa))}{\mathcal I_\beta^\eta \otimes \mathcal K(\ell_2(\kappa))} \rightarrow
  \frac{\widetilde{A_\eta}}{\mathcal I_\beta^\eta} \otimes \mathcal K(\ell_2(\kappa))$$
  be the natural isomorphism.  Then the composition $\psi \circ \sigma_{\theta, \beta}$
   is the map which sends $(a,\lambda)\otimes T \in \mathcal A_\theta$ to $((a,\lambda)+\mathcal I_\beta^\eta)\otimes
   T$. Therefore $\psi \circ \sigma_{\theta, \beta}=\widetilde{\sigma}_{\eta, \beta} \otimes id$ (where
   $\widetilde{\sigma}_{\eta, \beta}$ is the natural extension of $\sigma_{\eta, \beta}$ to 
   $\widetilde{\mathcal A_\eta}$ and  $id$ is the identity map on 
   $\mathcal K(\ell_2(\kappa))$). 
   Thus 
\begin{align}
 \nonumber   \mathcal I^{\theta}_{\beta+1}&= \sigma_{\theta,\beta}^{-1}( \I^{At}(\mathcal A_\theta/ \mathcal I^\theta_\beta)) \\
   \nonumber   &=  \sigma_{\theta,\beta}^{-1} \circ \psi^{-1} (\frac{ \mathcal I^{\eta}_{\beta+1}}{\mathcal I^\eta_\beta} \otimes  \mathcal K(\ell_2(\kappa))\big)\\
  \nonumber   &= (\widetilde{\sigma}_{\eta, \beta} \otimes id)^{-1} \big(\frac{ \mathcal I^{\eta}_{\beta+1}}{\mathcal I^\eta_\beta} \otimes  \mathcal K(\ell_2(\kappa))\big)\\
\nonumber  &=  \mathcal I^\eta_{\beta+1} \otimes \mathcal K(\ell_2 (\kappa)).
\end{align}   
 
 Assume $\beta\leq \eta$
    is a limit ordinal and
   the claim is true for all the ordinals below $\beta$.
   \begin{align}
 \nonumber \mathcal I^\theta_\beta &= \overline{\bigcup_{\alpha<\beta} \mathcal I^\theta_\alpha}
 = \overline{\bigcup_{\alpha<\beta} \mathcal I^\eta_\alpha \otimes  \mathcal K(\ell_2(\kappa))}\\
 \nonumber &= \overline{\bigcup_{\alpha<\beta} \mathcal I^\eta_\alpha} \otimes  \mathcal K(\ell_2(\kappa))
 = \mathcal I_\beta^\eta  \otimes  \mathcal K(\ell_2(\kappa)).
   \end{align}
Next step is to show that $ht(\mathcal A_\theta)= \theta= \eta+1$. First notice that
 \begin{align}
 \nonumber \mathcal I^{\theta}_{\theta}/ \mathcal I^\theta_\eta =
  \I^{At}\big(\mathcal A_\theta/ \mathcal I^\theta_\eta)&\cong \I^{At}\big(\frac{\widetilde{A_\eta}
   \otimes 
  \mathcal K(\ell_2 (\kappa))}{\mathcal A_\eta \otimes \mathcal K(\ell_2(\kappa))}\big) \\
  \nonumber &\cong \I^{At}\big(\frac{\widetilde{A_\eta}}{\mathcal A_\eta} \otimes \mathcal K(\ell_2(\kappa))\big)\\
\tag{Lemma \ref{atoms-tensor-product}}   &= \I^{At}\big(\frac{\widetilde{A_\eta}}{\mathcal A_\eta}\big)  \otimes \I^{At} \big(\mathcal K(\ell_2(\kappa))\big) \\
    \nonumber  &\cong \mathbb C \otimes  \mathcal K(\ell_2(\kappa)) \cong  \mathcal K(\ell_2(\kappa)).
\end{align}
Similar to the above we have
\begin{align}
 \nonumber   \mathcal I^{\theta}_{\theta}= (\widetilde{\sigma}_{\eta, \eta}\otimes id)^{-1} 
\big(\frac{\widetilde{A_\eta}}{\mathcal A_\eta} \otimes  \mathcal K(\ell_2(\kappa))\big)
=  \mathcal A_\theta.
\end{align}

Suppose $\theta$ is a limit ordinal. Let $\mathcal A_\theta$ be the inductive limit of the system 
$\{(\mathcal A_{\eta}, \pi_{\gamma, \eta}) : \eta\leq \gamma<\theta\}$. 
By repeatedly using the above claim we have
$$
\mathcal I^\theta_\beta = \mathcal I_\beta^\beta \otimes \bigotimes_{\beta<\gamma<\theta}
 \mathcal K(\ell_2(\kappa))
$$
for $\beta< \theta$.
For every $\eta<\theta$,
the map $\pi_{\theta, \eta}: \mathcal A_\eta \rightarrow \mathcal A_\theta$ defined by 
$$ \pi_{\theta,\eta}(a)= \lim_{\substack{\longrightarrow\\  \gamma< \theta }} \pi_{\gamma,\eta}(a) $$
is an embedding.
Assume $a \in \A_\theta$ and $\epsilon> 0$ are given. Then there are $\eta<\theta$ and $b\in \A_\eta$ such that
$$
\|\pi_{\theta,\eta}(b)  -a \|< \epsilon.
$$
Since $\A_\eta  = \mathcal I_{\eta}^{\eta}$, we have $\pi_{\theta,\eta}(b)\in \I_{\eta}^{\theta}$. Therefore 
$a \in \overline{\bigcup_{\eta<\theta} \I_{\eta}^{\theta}} = \I_\theta^\theta$. Thus $\A_\theta= \I_\theta^\theta$
 and $ ht(\mathcal A_\theta) = \theta$. Also 
$$
 \mathcal I^{\theta}_{\beta+1}/ \mathcal I^\theta_\beta \cong
  \frac{ \mathcal I^{\beta+1}_{\beta+1}}{\mathcal I^{\beta+1}_\beta} \otimes 
  \bigotimes_{\beta<\eta\leq\theta} \mathcal K(\ell_2(\kappa)) \cong  \mathcal K(\ell_2(\kappa)),
$$
if $\beta< \theta$.
 This completes the proof. 
\end{proof}

It is more difficult to increase the height of the algebra further without
increasing its width, i.e., keeping the previous algebra as
an essential ideal. In the following we focus on the case $\kappa=\omega$, and try to increase the height
of the algebra to uncountable ordinals while preserving the width $\omega$. We will see that the stability of an 
algebra can be employed  to guarantee  that it would be an essential ideal in the next algebra, helping us to maintain
the same width throughout the construction. The successor stage
 corresponds 
in the commutative case to dividing a locally
compact scattered Hausdorff space into infinitely many clopen noncompact subspaces
of the same height and one-point compactifying each part. This technique
is behind the standard example of a thin-tall locally compact space due to
Juhasz and Weiss (\cite{juhasz-weiss}).
For a scattered $C^*$-algebra $\A$, let $(\I_\beta(\A))_{\beta\leq ht(\A)}$ denote its Cantor-Bendixson sequence,

\begin{lemma}\label{extension-compact} Suppose that $\A$ a scattered 
 stable  $C^*$-algebra of height $\beta$.
Then there is a stable scattered $C^*$-algebra $\mathcal B$ of height $\beta+1$
containing $\A$ as essential ideal such that $\I_\beta(\B)=\A$ and $\B/\A\cong \K(\ell_2)$. 
\end{lemma}
\begin{proof} 
Since $\A$ is stable, we can identify $\A$ with $\A\otimes\K(\ell_2)$. Let 
$$\B=\widetilde{\A}\otimes \K(\ell_2).$$

As $\A$ is an essential ideal of $\widetilde{\A}$, we have that $\A\otimes\K(\ell_2)$
is an essential ideal of $\B$ (see the proof of Proposition
\ref{atoms-tensor-product}). It is also clear that $\B/\A\cong \C\otimes\K(\ell_2)\cong \K(\ell_2)$.
To calculate the Cantor-Bendixson derivatives of $\B$ we use Proposition \ref{cb-tensor}
and conclude that $\I_\alpha(\B)=\I_\alpha(\widetilde{\A})\otimes \K(\ell_2)$
for $\alpha\leq\beta$. Morover the isomorphism between 
$\A$ and  $\A\otimes\K(\ell_2)$ sends 
$\I_\alpha(\A)$ onto $\I_\alpha(\widetilde{\A})\otimes \K(\ell_2)$, by Proposition \ref{cb-tensor} 
and the
fact that an isomorphism must preserve the Cantor-Bendixson ideals.
\end{proof}

\begin{theorem}\label{exists-thin-tall} There is a thin-tall fully noncommutative $C^*$-algebra
(which is a subalgebra of $\B(\ell_2)$). 
\end{theorem}
\begin{proof} We construct an inductive limit $\A$ of stable scattered
separable fully noncommutative algebras $(\A_\alpha)_{\alpha<\omega_1}$
such that
\begin{itemize}
\item $\A_0 = \K(\ell_2 )$,
\item $\A_\alpha$ is an essential ideal of $\A_{\alpha+1}$ for every $\alpha<\omega_1$,
\item $\A_{\alpha+1}/\A_\alpha$ is $*$-isomorphic to $\K(\ell_2)$ for every $\alpha<\omega_1$,
\item $\A_\lambda$ is an inductive limit of $\A_\alpha$s for $\alpha<\lambda$ if
$\lambda$ is a countable limit ordinal.
\end{itemize} 
This is enough, since then by Lemma \ref{ideal-in-ideal} we have that
$\A_{\alpha+1}/\A_\alpha$ is an essential ideal in $\A/\A_\alpha$, and so by Proposition
\ref{essential-compact} we conclude that $\I^{At}(\A/\A_\alpha)=\A_{\alpha+1}/\A_\alpha$.

Given $\A_\alpha$ apply Lemma \ref{extension-compact} to get $\A_{\alpha+1}$. 
The resulting algebra satisfies the induction requirements, it is stable,
fully noncommutative, separable and scattered of height $\alpha+1$.

At countable limit ordinals take the inductive limit of the previously constructed chain
of ideals. It is clear that we obtain a fully noncommutative scattered algebra of appropriate height.
To prove that it is stable we use Corollary 2.3. (i) of \cite{rordam-stable} implying that
the countable inductive limit of separable stable algebras is stable.

The final algebra $\A$ is the inductive limit of the entire uncountable sequence. Note
that it can be embedded into $\B(\ell_2)$  applying Proposition \ref{embedding-scattered} as $\K(\ell_2)=\I^{At}_1(\A)$
is an essential ideal of $\A$ by Proposition \ref{scattered-are-atomic}
\end{proof}

Whether the algebra constructed in the theorem above is stable depends on 
the additional features of the construction. Clearly $\A\otimes\K(\ell_2)$ is stable, for $\A$ from the above theorem,
 so
the construction can be carried out in such a way that this isomorphism is respected.
However one can construct a fully noncommutative thin-tall scattered $C^*$-algebra which is not stable
(see \cite{thin-tall}). Note that the Cantor-Bendixson ideals of such
$C^*$-algebras are always stable by Lemma \ref{fnoncom-stable}. 
This provides examples of  $C^*$-algebras which are nonstable and with no maximal stable ideal.
 This follows from the fact that in fully commutative scattered
$C^*$-algebras all ideals are among the Cantor-Bendixson ideals by Lemma \ref{all-ideals}.

\begin{theorem}\label{long-thin-tall} There is 
a fully noncommutative scattered algebra $\A\subseteq \B(\ell_2)$ 
which is an inductive limit  of a sequence $(\I_\alpha)_{\alpha<\lambda}$
of its essential ideals  for a limit ordinal $\lambda\leq \cc^+$  of uncountable cofinality
such that $\I_{\alpha+1}/\I_\alpha$
is isomorphic to $\K(\ell_2)$ for each $\alpha<\lambda$, with each $\I_\alpha$ stable
but $\A$ nonstable.  In particular stability of $C^*$-algebras is not preserved by 
uncountable inductive limits and
there are nonstable $C^*$-algebras without a stable ideal which is maximal among stable ideals.
\end{theorem}
\begin{proof} 
Perform the recursive construction as in the proof 
of Theorem \ref{exists-thin-tall} 
 obtaining $\A_\alpha$s which are
fully noncommutative, scattered of height $\alpha$ and 
form an increasing continuous  chain of essential ideals. The recursive construction has
length $\lambda$ where 
 $\lambda$ be the first ordinal
at which $\A_\lambda$
is not stable. Note that for all $\alpha<\lambda$ we can continue the 
construction from the proof 
of Theorem \ref{exists-thin-tall}  beyond $\alpha$ as the algebras $A_\alpha$ are stable by the minimality of $\lambda$
and so the  following successor step is possible and at limit ordinals (possibly uncountable) we take the inductive limits.

First we prove that there is such $\lambda<\cc^+$. Suppose not,  and let us
derive a contradiction. Then
we can produce $\A_{\cc^+}$ and consider
 $\I^{At}(\A_{\cc^+})$, which is $^*$-isomorphic to $\K(\ell_2)$ and it is an essential ideal  
of $\A_{\cc^+}$.  It follows from Proposition \ref{embedding-scattered} that $\A_{\cc^+}$ embeds into the multiplier algebra of $\K(\ell_2)$,
which is $\B(\ell_2)$. It follows that the density of
$\A_{\cc^+}$ is at most continuum, and so the height $ht(\A_{\cc^+})<(2^\omega)^+$
as successor cardinals are regular,
a contradiction. 

Now let us see that $\lambda$ is a limit ordinal
of uncountable cofinality. If $\lambda$ is a successor,
then the resulting algebra is 
stable  by Lemma \ref{extension-compact}. At limit
ordinals of countable cofinality we can use Corollary 4.1 of 
from \cite{rordam-h}, which states that an inductive limit of a sequence 
of $\sigma$-unital stable $C^*$-algebras is stable. 
Note that algebras of the form $\B=\widetilde{\A}\otimes \K(\ell_2)$
are $\sigma$-unital, because $(1\otimes p)$ form a countable approximate
unit of them, where $p$ runs through  finite dimensional projections
onto spaces spanned by first $n\in \N$ vectors of some fixed orthonormal basis of $\ell_2$.
The countable cofinality of $\lambda$ ensures that 
our inductive limit $\A_\lambda$ is also an inductive limit of a sequence
of algebras $(\A_{\lambda_n})_{n\in \N}$, where $(\lambda_n)_{n\in \N}$
is cofinal in $\lambda$. This implies that $\A_\lambda$ is stable, which is  a contradiction.
 Therefore $\lambda$ is as required.

The absence of a maximal stable ideal in $\A_\lambda$ follows from
the minimality of $\lambda$ and 
the fact that all ideals in a  fully noncommutative scattered $C^*$-algebras
are among $\A_\alpha$s for $\alpha<\lambda$
(Lemma \ref{all-ideals}).
\end{proof}

Our final example is a noncommutative version of the $\Psi$-space (cf. \cite{hrusak}):

\begin{theorem}\label{exists-psi}
There is a stable and fully noncommutative scattered $C^*$-algebra $\A$ of height 2
with $\I^{At}(\A)\cong \K(\ell_2)$ and $\A/\I^{At}(\A)\cong \K(\ell_2(\cc))$.
\end{theorem}
\begin{proof}
Let $(A_\xi: \xi<\cc)$ be an almost disjoint family of subsets of $\N$, 
that is such a family of infinite subsets of $\N$ that $A_\xi\cap A_\eta$ is finite for 
any two distinct $\xi, \eta<\cc$. Define in $\B(\ell_2)$ orthogonal projections
$P_\xi$ onto the spaces $span(\{e_n: n\in A_\xi\})$ where $(e_n)_{n\in \N}$ 
is a  fixed orthogonal basis of $\ell_2$. 
Let $\sigma_\xi: A_0\rightarrow A_\xi$ be bijections and
$T_{\xi, 0}\in \B(\ell_2)$ be corresponding partial isometries i.e.,
$T_{\xi, 0}(e_n)=e_{\sigma_\xi(n)}$ if $n\in A_0$ and $T_{\xi, 0}(e_n)=0$
otherwise.  Note that $T_{\xi, 0}^*(e_n)=e_{\sigma^{-1}_\xi(n)}$ if $n\in A_\xi$ 
and $T_{\xi, 0}(e_n)=0$
otherwise. 
For all $\xi, \eta<\kappa$
 define 
$$T_{\eta,  \xi}=T_{\eta,0}T_{\xi, 0}^*.$$ 
That is  $T_{\eta, \xi}(e_n)=e_{\sigma_\eta\circ\sigma^{-1}_\xi(n)}$ if $n\in A_\xi$ 
and $T_{\eta, \xi}(e_n)=0$
otherwise. 
Note that $T_{\eta, \xi}=P_\eta T_{\eta, \xi}=T_{\eta, \xi}P_\xi=P_\eta T_{\eta, \xi}P_\xi$.

Let $\mathcal A$ be a $C^*$-algebra generated in $\B(\ell_2)$ by
$\{T_{\eta, \xi}: \xi, \eta<\cc\}$ and the compact operators. 
As $\K(\ell_2)$ is an essential ideal in $\B(\ell_2)$, it is essential in $\A$ and so
by Proposition \ref{essential-compact} we have $\I^{At}(\A)=\K(\ell_2)$. Now consider the quotient
$\A/\K(\ell_2)$. We have 
$$[T_{\eta, \xi}]^*_{\K(\ell_2)}=[T_{\xi, \eta}]_{\K(\ell_2)}\leqno (1)$$
as $T_{\eta, \xi}^*=T_{\xi, \eta}$. Moreover
$$[T_{\beta, \alpha}]_{\K(\ell_2)}[T_{\xi, \eta}]_{\K(\ell_2)}
=\delta_{\alpha, \xi}[T_{\beta, \eta}]_{\K(\ell_2)}\leqno (2)$$
where $\delta_{\alpha,\xi}=1$ if $\xi=\alpha$ and 
$\delta_{\alpha,\xi}=0$  otherwise. This is checked directly if $\alpha=\xi$ and
for $\alpha\not=\xi$ we use 
$T_{\beta, \alpha}T_{\xi, \eta}=T_{\beta, \alpha}P_\alpha P_\xi T_{\xi, \eta}$
and the fact that $P_\alpha P_\xi$ is the projection on  a finite
dimensional space $span(\{e_n: n\in A_\alpha\cap A_\xi\})$, since $A_\xi$s are
almost disjoint. It follows that $T_{\beta, \alpha}T_{\xi, \eta}$ is compact if $\alpha\not=\xi$.
It is well-known that $C^*$ algebras having nonzero generators
satisfying (1) and (2) are isomorphic to the algebra of compact operators
on the Hilbert space $\ell_2(\cc)$. This  is for example checked in Section 2.3.
of \cite{psi-space} and this completes the proof.

\end{proof}

As in the case of a thin-tall $C^*$-algebra the above $\Psi$-algebra may or may not
be stable. An example exhibiting very strong nonstability (its multiplier algebra is isomorphic to the minimal
unitization) is
constructed in \cite{psi-space}.

\bibliographystyle{amsplain}

\begin{thebibliography}{20}

\bibitem{akemann-doner} C. Akemann, J.  Doner, 
\emph{A nonseparable $C^*$-algebra with only separable abelian $C^*$-subalgebras}.
Bull. London Math. Soc. 11 (1979), no. 3, 279--284.

%\bibitem{antonevich} A. Antonevich, \emph{Linear functional equations. Operator approach}. 
%Translated from the 1988 Russian original by Victor 
%Muzafarov and Andrei Iacob. Operator Theory: Advances and Applications, 83. Birkhauser Verlag, Basel, 1996.

\bibitem{invitation} W. Arveson,
\emph{An invitation to $C^*$-algebras},
Graduate Texts in Mathematics, No. 39. Springer-Verlag, New York-Heidelberg, 1976.

\bibitem{baumgartner}  J. Baumgartner, \emph{Almost-disjoint sets, the dense 
set problem and the partition calculus}. Ann. Math. Logic 9 (1976), no. 4, 401--439. 

\bibitem{bs} J.Baumgartner, S.  Shelah, 
\emph{Remarks on superatomic Boolean algebras}.
Ann. Pure Appl. Logic 33 (1987), no. 2, 109--129. 

\bibitem{tristan} T. Bice, P. Koszmider,  
\emph{A note on the Akemann-Doner and Farah-Wofsey constructions},
 Proc. Amer. Math. Soc. 145 (2017), no. 2, 681--687.

\bibitem{aiau} T. Bice, P. Koszmider, \emph{C*-algebras with and without $\ll$-increasing approximate units},
Arxiv: 1707.09287.

\bibitem{blackadar-af} B. Blackadar, \emph{Traces on simple AF $C^*$-algebras}.
 J. Funct. Anal. 38 (1980), no. 2, 156--168. 

\bibitem{Blackadar} B. Blackadar, \emph{Operator algebras: theory of C*-algebras and von Neumann algebras}, Vol. 122. Springer Science $\&$ Business Media, 2006.

\bibitem{Blanchard-Kirchberg} E. Blanchard, and E. Kirchberg, \emph{Non-simple purely infinite 
$C^*$-algebras: the Hausdorff case}. Journal of Functional Analysis 207.2 (2004): 461-513.

\bibitem{bonnet-rubin} R. Bonnet, M. Rubin, \emph{A thin-tall Boolean algebra which is
 isomorphic to each of its uncountable subalgebras}. Topology Appl. 158 (2011), no. 13, 1503--1525.

\bibitem{christina} C. Brech, P. Koszmider, \emph{Thin-very tall compact scattered spaces which are 
hereditarily separable}. Trans. Amer. Math. Soc. 363 (2011), no. 1, 501--519. 

\bibitem{BDF} L. Brown, R.  Douglas, P.  Fillmore, \emph{Extensions of $C^*$-algebras and K-homology}.
 Ann. of Math. (2) 105 (1977), no. 2, 265--324.

\bibitem{brown}  L. Brown, G. Pedersen, \emph{$C^*$-algebras of real rank zero}. 
J. Funct. Anal. 99 (1991), no. 1, 131--149. 

\bibitem{chu-rn} C. Chu, \emph{A note on scattered $C^*$-algebras and the Radon-Nikodym property},
J. London Math. Soc. 24 (1981), 533--536.

\bibitem{chu-crossed} C.  Chu, \emph{Crossed products of scattered $C^*$-algebras},
 J. London Math. Soc. (2) 26 (1982), no. 2, 317--324.

\bibitem{chu-banach-saks}
C. Chu,
\emph{The weak Banach-Saks property in $C^*$-algebras.}
J. Funct. Anal. 121 (1994), no. 1, 1--14.

\bibitem{davidson} K. Davidson, \emph{$C^*$-algebras by example}. 
Fields Institute Monographs, 6. American Mathematical Society, Providence, RI, 1996.

\bibitem{day} G. Day, \emph{Superatomic Boolean algebras}. Pacific J. Math. 23 1967 479--489.

\bibitem{dgz} R.  Deville, G.  Godefroy, V. Zizler, \emph{Smoothness and renormings in Banach spaces}.
 Pitman Monographs and Surveys in Pure and Applied Mathematics, 64.
 Longman Scientific \& Technical, Harlow; copublished in the United States with John Wiley \& Sons, Inc., New York, 1993.

\bibitem{vector-measures} J. Diestel, J. Uhl, \emph{Vector measures}. 
With a foreword by B. J. Pettis. Mathematical Surveys, No. 15. American Mathematical Society, Providence, R.I., 1977.

\bibitem{dow-simon} A.  Dow, P. Simon, \emph{Thin-tall Boolean algebras and their 
automorphism groups}. Algebra Universalis 29 (1992), no. 2, 211--226. 



\bibitem{engelking} R. Engelking, \emph{General topology}. 
Translated from the Polish by the author. Second edition. 
Sigma Series in Pure Mathematics, 6. Heldermann Verlag, Berlin, 1989.

%\bibitem{exel} R. Exel, M. Laca and J. Quigg, \emph{Partial dynamical systems and $C^*$-algebras generated by partial isometries},  J. Operator Theory, 47 (2002), 169--186.

\bibitem{farah-katsura}  I. Farah, T. Katsura, \emph{Nonseparable UHF algebras I: Dixmier's problem}. 
Adv. Math. 225 (2010), no. 3, 1399--1430. 

\bibitem{Farah-model-theory} I. Farah, B. Hart, M. Lupini, L. Robert, A. Tikuisis, A. Vignati, W. Winter, \emph{Model theory
of $C^*$-algebras}. arXiv:1602.08072. 

\bibitem{psi-space} S. Ghasemi, P. Koszmider, \emph{An extension
of compact operators by compact operators with no nontrivial multipliers}. Arxiv: 1609.04766.

\bibitem{thin-tall} S. Ghasemi, P. Koszmider \emph{A non-stable C*-algebra with an elementary essential composition series }. 
Arxiv: 1712.02090.

\bibitem{halmos}  S. Givant, P.  Halmos, \emph{Introduction to Boolean algebras}. 
Undergraduate Texts in Mathematics. Springer, New York, 2009.

\bibitem{biorthogonal-hajek}  P. Hajek, V. Montesinos Santalucia, J. Vanderwerff, V.  Zizler, 
\emph{Biorthogonal systems in Banach spaces}. CMS Books in Mathematics/Ouvrages de Mathematiques 
de la SMC, 26. Springer, New York, 2008.

%\bibitem{Glimm-Kadison} J. G. Glimm,  R. V.  Kadison, \emph{ Unitary operators in C*-algebras}. Pacific J. Math, 10,  547-556 (1960).

%\bibitem{Guichardet} A. Guichardet, \emph{Tensor product of $C^*$-algebras}, Lecture Note series 12, Part I
%(Aarhus University), 1969.

%\bibitem{kappa-asplund-universal} P. Hajek, G. Lancien, V. Montesinos, \emph{Universality of Asplund spaces}.
%Proc. Amer. Math. Soc. 135 (2007), no. 7, 2031--2035.

\bibitem{rordam-h} J. Hjelmborg, M. R{\o}rdam, 
\emph{On stability of $C^*$-algebras}. 
J. Funct. Anal. 155 (1998), no. 1, 153--170. 

\bibitem{huruya} T. Huruya,
\emph{A spectral characterization of a class of $C^*$-algebras}.
Sci. Rep. Niigata Univ. Ser. A No. 15 (1978), 21--24.

\bibitem{huruya2} T. Huruya, \emph{The Second Dual Of A Tensor Product Of $ C^{*} $-algebras, II}.
 Science Reports of Niigata University. Series A, Mathematics 11 (1974): 21-23.

\bibitem{hrusak}  M. Hrusak, \emph{Almost disjoint families and topology}. Recent progress in general topology. III, 601--638, Atlantis Press, Paris, 2014.

\bibitem{jensen1} H. Jensen, \emph{Scattered $C^*$-algebras}. Math. Scand. 41 (1977), no. 2, 308--314.

\bibitem{jensen2} H. Jensen, \emph{Scattered $C^*$-algebras. II}. Math. Scand. 43 (1978), no. 2, 308--310 (1979).

\bibitem{jensen3} H. Jensen, \emph{Scattered $C^*$-algebras with almost finite spectrum}.
 J. Funct. Anal. 50 (1983), no. 2, 127--132.

\bibitem{juhasz-weiss} I. Juhasz, W. Weiss, 
\emph{On thin-tall scattered spaces}.
Colloq. Math. 40 (1978/79), no. 1, 63--68.

\bibitem{just}  W. Just, 
\emph{Two consistency results concerning thin-tall Boolean algebras}.
Algebra Universalis 20 (1985), no. 2, 135--142.

\bibitem{koepke} P. Koepke, J.  Martinez,
\emph{Superatomic Boolean algebras constructed from morasses}.
J. Symbolic Logic 60 (1995), no. 3, 940--951. 

\bibitem{koppelberg-minimal} S.  Koppelberg, \emph{Minimally generated Boolean algebras}. Order 5 (1989), 
no. 4, 393--406.

\bibitem{koppelberg-handbook} S. Koppelberg, \emph{Handbook of Boolean algebras}. Vol. 1. Edited by J. Donald Monk and Robert Bonnet. North-Holland Publishing Co., Amsterdam, 1989. 

\bibitem{kusuda-crossed} M. Kusuda, \emph{A characterization of scattered $C^*$-algebras and
its applications to $C^*$-crossed products}, J. Operator Theory 63 (2010), 417--424.

\bibitem{kusuda-af}
M. Kusuda,
\emph{$C^*$-algebras in which every $C^*$-subalgebra is AF}
Q. J. Math. 63 (2012), no. 3, 675--680.

\bibitem{lancien}  G. Lancien, \emph{A survey on the Szlenk index and some of its applications}. 
RACSAM. Rev. R. Acad. Cienc. Exactas Fis. Nat. Ser. A Mat. 100 (2006), no. 1-2, 209--235. 

\bibitem{lin} H. X. Lin, \emph{The structure of quasimultipliers of $C^*$-algebras}. 
Trans. Amer. Math. Soc. 315 (1989), no. 1, 147--172. 

\bibitem{lj}  W. Johnson, J.  Lindenstrauss, \emph{Some remarks on weakly
 compactly generated Banach spaces}. Israel J. Math. 17 (1974), 219--230. 

\bibitem{loring}  T. Loring, \emph{Lifting solutions to perturbing problems in $C^*$-algebras}.
Fields Institute Monographs, 8. American Mathematical Society, Providence, RI, 1997.

\bibitem{martinez-forcing} J. C. Martinez,
\emph{A consistency result on thin-tall superatomic Boolean algebras}.
Proc. Amer. Math. Soc. 115 (1992), no. 2, 473--477. 

\bibitem{martinez-jsl}  J. C. Martinez, \emph{A consistency result on thin-very tall Boolean algebras}.
 Israel J. Math. 123 (2001), 273--284.

%\bibitem{martinez-survey} J. C.  Martinez, \emph{Cardinal sequences for superatomic Boolean algebras}. Infinity, computability, and metamathematics, 273--284, Tributes, 23, Coll. Publ., London, 2014.

\bibitem{mazurkiewicz-sierpinski} S. Mazurkiewicz, W. Sierpinski, \emph{Contribution a la
topologie des ensembles denombrables}, Fund Math 1, 1920, 17--27.

\bibitem{mostowski-tarski} A. Mostowski, A. Tarski, \emph{Booleshe Ringe mit ordneter basis}.
Fund. Math.  32,  69--86.

\bibitem{murphy}  G. Murphy, \emph{$C^*$-algebras and operator theory}.
Academic Press, Inc., Boston, MA, 1990.

\bibitem{mrowka} S. Mrowka. \emph{Some set-theoretic constructions in topology}. 
Fund. Math., 94 (2), 83--92, 1977.

\bibitem{negrepontis}  S. Negrepontis, \emph{Banach spaces and topology}. 
Handbook of set-theoretic topology, 1045--1142, North-Holland, Amsterdam, 1984. 

\bibitem{ostaszewski} A. Ostaszewski, \emph{On countably compact, perfectly normal spaces.}
 J. London Math. Soc. (2) 14 (1976), no. 3, 505--516. 

\bibitem{pedersen}  G. K. Pedersen, \emph{$C^*$-algebras and their automorphism groups}. 
London Mathematical Society Monographs No. 14, Academic Press, 1979.

\bibitem{pelczynski-semadeni} A.  Pe\l czy\'nski, Z.  Semadeni, \emph{Spaces of continuous functions. III.
Spaces $C(\Omega)$ for $\omega$ without perfect subsets}. Studia Math. 18 1959 211--222.

\bibitem {pol}  R. Pol, \emph{A function space $C(X)$
 which is weakly Lindel\"of but not weakly compactly generated}. 
Studia Math. 64 (1979), no. 3, 279--285.

\bibitem{roitmanhandbook}  J. Roitman, \emph{Superatomic Boolean algebras}.
Handbook of Boolean algebras, Vol. 3, 719--740, North-Holland, Amsterdam, 1989.

\bibitem{roitman-factorial} J. Roitman, \emph{$\omega!$ can be a nontrivial automorphism group}. Proc. Amer. Math. Soc. 112 (1991), no. 3, 623--628.

\bibitem{roitman-uncountable} J. Roitman, \emph{Uncountable autohomeomorphism 
groups of thin-tall locally compact scattered spaces}. Algebra Universalis 29 (1992), no. 3, 323--331. 



\bibitem{rordam-stable} M. R{\o}rdam, \emph{Stable $C^*$-algebras}. Operator algebras and applications, 
177--199, Adv. Stud. Pure Math., 38, Math. Soc. Japan, Tokyo, 2004.

\bibitem{rudin} W. Rudin,
\emph{Continuous functions on compact spaces without perfect subsets.}
Proc. Amer. Math. Soc. 8 (1957), 39--42.

\bibitem{sakai} S. Sakai, C*-algebras and W*-algebras, Classics in Mathematics, Springer, 1998.

%\bibitem{rubin-koppelberg} M. Rubin, S. Koppelberg, \emph{A superatomic Boolean algebra with few automorphisms}. Arch. Math. Logic 40 (2001), no. 2, 125--129.

\bibitem{schreier} J. Schreier, \emph{Ein Gegenbeispiel zur Theorie der schwachen Convergenz},
Studia Math.  2, 1930, 58--62.

\bibitem{simon}  P. Simon, \emph{A compact Frechet space
 whose square is not Frechet}. Comment. Math. Univ. Carolin. 21 (1980), no. 4, 749--753.

\bibitem{tomiyama} J. Tomiyama, \emph{A characterization of C*-algebras whose conjugate spaces are separable}. Tohoku Mathematical Journal, Second Series 15.1 (1963): 96-102.

%\bibitem{shelah-steprans} S. Shelah, J.  Steprans, \emph{Homogeneous almost disjoint families}. Algebra Universalis 31 (1994), no. 2, 196--203.

\bibitem{tall}  F. D. Tall, \emph{Set-theoretic consistency results and topological theorems concerning
the normal Moore space conjecture and related problems},
 Thesis (Ph.D.) -The University of Wisconsin - Madison. 1969. 

\bibitem{Williams} D. P. Williams,  \emph{Crossed products of $C^*$-algebras}. No. 134. American Mathematical Soc., 2007.

\bibitem{wofsey} E. Wofsey, \emph{$\wp(\omega)/fin$ and projections in the Calkin algebra}, Proc. Amer. Math. Soc.
136 (2008), no. 2, 719--726.

%\bibitem{wojtaszczyk-asplund} P. Wojtaszczyk, \emph{On separable Banach
%spaces containing all separable reflexive Banach spaces}. Studia Math. 37 (1970/71), 197--202.

\bibitem{wojtaszczyk} P.  Wojtaszczyk, \emph{On linear properties of separable conjugate spaces of $C^*$-algebras}.
Studia Math. 52 (1974), 143--147.

%\bibitem{yost} D.  Yost, \emph{Asplund spaces for beginners}. Selected papers from the 21st Winter School on Abstract Analysis (Podebrady, 1993). Acta Univ. Carolin. Math. Phys. 34 (1993), no. 2, 159--177.

\end{thebibliography}

\end{document}